\documentclass[11pt]{amsart}

\usepackage{amsmath}
\usepackage{amssymb}
\usepackage{amsthm}
\usepackage{color}  
\swapnumbers     
\usepackage{enumerate}
\usepackage{tikz-cd}

\title[$p$-adic limits of topological invariants]
 {On $p$-adic limits of topological invariants} 

\author[S. Kionke]{Steffen Kionke}

\address{Karlsruher Institut f\"ur Technologie \\ Fakult\"at f\"ur Mathematik \\
  Institut f\"ur Algebra und Geometrie \\ Englerstr.2 \\
  76131 Karlsruhe \\ Germany.}

\email{steffen.kionke@kit.edu}

\thanks{The research was funded by the Deutsche Forschungsgemeinschaft
  (DFG, German Research Foundation) -  338540207.}

\date{\today}

\numberwithin{equation}{section}

\theoremstyle{plain}
\newtheorem{theorem}[subsection]{Theorem}
\newtheorem*{theorem*}{Theorem}
\newtheorem*{proposition*}{Proposition}
\newtheorem*{lemma*}{Lemma}
\newtheorem{lemma}[subsection]{Lemma}

\newtheorem*{corollary*}{Corollary}
\newtheorem{proposition}[subsection]{Proposition}

\theoremstyle{definition}
\newtheorem{definition}[subsection]{Definition}
\newtheorem*{conjecture}{Conjecture}
\newtheorem{remark}[subsection]{Remark}
\newtheorem*{remark*}{Remark}
\newtheorem{example}[subsection]{Example}
\newtheorem*{example*}{Example}
\newtheorem{question}[subsection]{Question}
\newtheorem*{question*}{Question}

\DeclareMathOperator{\Hom}{Hom}

\DeclareMathOperator{\Aut}{Aut}
\DeclareMathOperator{\GL}{GL}
\DeclareMathOperator{\Ext}{Ext}
\DeclareMathOperator{\chr}{char}
\DeclareMathOperator{\Ind}{ind}
\DeclareMathOperator{\Irr}{Irr}
\DeclareMathOperator{\infl}{infl}
\DeclareMathOperator{\Gal}{Gal}
\DeclareMathOperator{\Res}{res}
\DeclareMathOperator{\Or}{Or}
\DeclareMathOperator{\adm}{adm}
\DeclareMathOperator{\im}{im}
\DeclareMathOperator{\hyphendim}{-dim}

\DeclareMathOperator{\tors}{tors}
\DeclareMathOperator{\sign}{sign}
\newcommand{\pdim}{p\hyphendim}
\newcommand{\pcard}{\#_p}

\providecommand{\isomorph}{\stackrel{\simeq}{\longrightarrow}}
\newcommand{\normal}{\trianglelefteq}

\providecommand{\bbN}{\mathbb{N}}

\providecommand{\bbQ}{\mathbb{Q}}
\providecommand{\bbZ}{\mathbb{Z}}
\providecommand{\bbF}{\mathbb{F}}
\providecommand{\bbC}{\mathbb{C}}
\providecommand{\admK}{\mathsf{K}^{\mathsf{adm}}_0}

\newcommand{\cga}[2]{ #1[\![#2]\!]}
\newcommand{\up}[1]{\:^{#1} }
\newcommand{\alg}[1]{\overline{#1} }

\providecommand{\admA}{\mathcal{A}_p^{\mathsf{adm}}}
\providecommand{\calU}{\mathcal{U}}
\providecommand{\calM}{\mathcal{M}}

\providecommand{\calF}{\mathcal{F}}
\providecommand{\fu}{\mathfrak{u}}
\providecommand{\gl}{\mathfrak{gl}}

\renewcommand{\epsilon}{\varepsilon}
\renewcommand{\phi}{\varphi}

\begin{document}

\begin{abstract}
  The purpose of this article is to
  define and study new invariants of topological spaces: the $p$-adic Betti numbers and the $p$-adic torsion.
  These invariants take values in the $p$-adic numbers and
  are constructed from a virtual pro-$p$ completion of
  the fundamental group.
  The key result of the article is an approximation theorem
  which shows that the $p$-adic invariants are in limits of their classical analogues.
  This is reminiscent of L\"uck's approximation theorem for $L^2$-Betti numbers.
  
  After an investigation of basic properties and examples 
  we discuss the $p$-adic analog of the Atiyah conjecture:
  When do the $p$-adic Betti numbers take integer values?
  We establish this for a class of spaces and discuss applications to cohomology growth.
\end{abstract}

\maketitle

\section{Introduction}
Fix a prime number $p$.
Let $X$ be a finite CW-complex with fundamental group $\Gamma = \pi_1(X,x_0)$.  The purpose of
this article is to construct and study $p$-adic valued invariants of $X$ using virtual pro-$p$
completions of its fundamental group.  A \emph{virtual pro-$p$ completion} of $\Gamma$ is a
homomorphism $\phi\colon \Gamma \to G$ into a profinite group $G$ which has an open pro-$p$
subgroup (cf.~\S\ref{def:v-pro-p-completion}).
The first family of invariants to be defined are the
\emph{$p$-adic Betti numbers}. The $j$-th $p$-adic Betti number of $X$ is a $p$-adic integer  
$$b_j^{[p]}(X;\phi,k) \in \bbZ_p$$
 which depends on a virtual pro-$p$ completion $\phi$ and a field $k$ which is
 either the field of rational numbers $\bbQ$ or a finite field of characteristic $\ell \neq p$.
 The $p$-adic Betti numbers
are homotopy invariants in the natural sense: If $f \colon Y \to X$
is a homotopy equivalence,
then $b_j^{[p]}(X;\phi,k) = b_j^{[p]}(Y;  \phi\circ f_*, k)$ where $f_*\colon \pi_1(Y,y_0) \to \pi_1(X,x_0)$
is the induced isomorphism of fundamental groups.
The second invariant to be defined here is the \emph{$p$-adic torsion} of $X$. The $j$-th $p$-adic
torsion of $X$ is a $p$-adic integer (or $\infty$)
$$t_j^{[p]}(X;\phi,R) \in \bbZ_p \cup \{\infty \} $$
which depends on a virtual pro-$p$ completion $\phi$ and a commutative ring $R$ in which $p$ is invertible.
Just as $p$-adic Betti numbers the $p$-adic torsion is a homotopy invariant.
The definition of $p$-adic invariants does not depend on the cellular structure and
generalizes to connected, locally path-connected and semilocally simply-connected spaces.

\subsubsection*{The Approximation Theorem}
The central result in this article is an approximation theorem which shows that the $p$-adic
invariants are $p$-adic limits of their classical analogues.
Recall that the fundamental group $\Gamma$ acts on the universal covering $\widetilde{X}$
of $X$ via deck transformations such that $X \cong \widetilde{X}/\Gamma$ and every finite index subgroup
$\Delta \leq_{f.i.} \Gamma$ gives rise to a finite sheeted covering space $\widetilde{X}/\Delta$ of $X$.
\begin{theorem}\label{thm:intro-approximation}
  Let $X$ be a finite CW-complex and let $\phi \colon \Gamma \to G$ be a virtual pro-$p$ completion
  of its fundamental group.  For every chain $G_1 \supseteq G_2 \supseteq G_3 \dots$ 
  of open normal subgroups in $G$ which satisfies $\bigcap_{n \in \bbN} G_n = \{1\}$,
  the following assertions hold:
  \begin{align*}
    \text{(i) } &\lim_{n \to \infty} b_j(\widetilde{X}/\phi^{-1}(G_n); k) = b_j^{[p]}(X;\phi,k),\\
  \text{(ii) } &\lim_{n \to \infty} \bigl| \tors H^j(\widetilde{X}/\phi^{-1}(G_n); R) \bigr| = t_j^{[p]}(X;\phi,R),\end{align*}
  where the limits are taken in $\bbZ_p \cup \{\infty\}$.
  Here $|\tors H^j(Y;R)|$ denotes cardinality of the $R$-torsion part in the $j$-th cohomology.
\end{theorem}
This result is a special case of Theorem \ref{thm:approximation} below, which is significantly more general.
It is worth noting that the theorem, in particular,
includes the statement that the limits exist and, for this introduction, it is receivable to consider this the definition of
the $p$-adic invariants.

\subsubsection*{Comparison with  other Betti numbers}
The approximation theorem allows to transfer some properties of ordinary Betti numbers to
$p$-adic Betti numbers.  For instance, straightforward applications of the Approximation
Theorem yield a K\"unneth formula (see \S\ref{par:products}) and a form of Poincar\'e duality
(see Proposition~\ref{prop:poincare-duality}).

In other aspects, however, the nature of $p$-adic Betti
numbers is quite different from the nature of ordinary Betti numbers.
For example, the $p$-adic Betti numbers are invariant under passing to finite sheeted covering spaces (see \S\ref{par:virtual-invariance}) and the $p$-adic Euler characteristic always vanishes if $\phi$ has an infinite image (see \S\ref{par:eulerchar}).

It seems more appropriate to consider $p$-adic invariants to be analogues of
$L^2$-invariants. From this perspective the $p$-adic Approximation Theorem is the analogue of
L\"uck's approximation theorem for $L^2$-Betti numbers which states that
$$\lim_{n\to \infty} \frac{b_j(\widetilde{X}/\Gamma_n;\bbQ)}{|\Gamma:\Gamma_n|} = b_j^{(2)}(X)$$
for every chain of finite index normal subgroups $\Gamma_n \normal_{f.i.} \Gamma$ with trivial intersection; see \cite{Lück94}.
To a certain extend our results on $p$-adic approximation go beyond the known $L^2$-approximation results.
For instance, it is currently unknown
whether an analogue of L\"uck's theorem holds for Betti numbers with coefficients in a finite field
(see Conjecture 3.4 in \cite{LückSurvey16}), whereas
the $p$-adic Approximation Theorem is valid for $k = \bbF_\ell$ for all primes $\ell \neq p$.
Similarly, the $p$-adic Approximation Theorem applies to torsion cohomology,
whereas the approximation of $L^2$-torsion by torsion homology is a fundamental open problem;
see Conjecture 8.9 in~\cite{LückSurvey16}.
Anton Clau{\ss}nitzer discussed another approach
towards an $p$-adic form of L\"uck's approximation theorem in his Ph.D. thesis \cite{Claussnitzer18}.
His methods are based on $p$-adic operator algebras
but seem to be to restrictive to apply to the topological setting.

In spite of the analogy, the $p$-adic Betti numbers behave very much like antagonists of the $L^2$-Betti numbers.
For example, the $L^2$-Betti numbers distinguish free groups of different rank, whereas
the $p$-adic Betti numbers of the free group $F_r$
do not depend on the rank $r$; more precisely
$$b_0^{[p]}(F_r;\phi,k) = b_1^{[p]}(F_r;\phi,k) = 1$$
whenever $\phi$ has an infinite image; see \S\ref{par:free-groups}. Similarly, the $p$-adic
Betti numbers do not see the genus of a surface; see \S\ref{par:surfaces}.  On the other hand,
the $p$-adic Betti numbers of free abelian groups coincide with the ordinary Betti numbers
(see \S\ref{par:tori}), whereas the $L^2$-Betti numbers of every infinite amenable group
vanish; see Theorem 0.2 in \cite{CheegerGromov86}.

The Approximation Theorem establishes a relationship between the $j$-th $p$-adic Betti number
of $X$ and the $j$-th virtual Betti number of $X$, which is defined as the supremum of
rational Betti numbers $vb_j(X) = \sup_{Y} b_j(Y;\bbQ)$ over all finite sheeted covering
spaces $Y$ of $X$.  In fact, suppose that $vb_j(X)$ is finite, then the sequence approximating
$b_j^{[p]}(X;\phi,k)$ has to stabilize at some natural number between $0$ and $vb_j(X)$.  Even
though the virtual Betti numbers often take the value $\infty$, there are interesting families
of groups and spaces with finite virtual Betti numbers, e.g.\ finitely presented
nilpotent-by-abelian groups have finite first virtual Betti numbers; see
\cite{BridsonKochloukova15}.

\subsubsection*{The values of $p$-adic Betti numbers}
In all examples we studied the $p$-adic Betti numbers are in fact integers.
This comes as a surprise, since the construction of $p$-adic invariants really uses $p$-adic analysis.
In fact, the $p$-adic torsion can take transcendental values; see \S\ref{par:example-torsion}.
The following question is intruiging.
\begin{question}
  Under which conditions on $X$ and $\phi$ are the $p$-adic Betti numbers $b_j^{[p]}(X;\phi,k)$
  in $\bbZ$?
\end{question}
Our list of examples is simply too short to propose a conjectural answer.
We will refer to this question as the \emph{$p$-adic Atiyah question},
since it is reminiscent of the Atiyah conjecture for $L^2$-Betti numbers.
\begin{conjecture}
  Let $X$ be a finite connected CW-complex with fundamental group $\Gamma$.
  If $\Gamma$ is torsion-free, then $b_j^{(2)}(X) \in \bbZ$.
\end{conjecture}
This conjecture goes back to a question of Atiyah concerning the values of $L^2$-Betti numbers
of compact Riemannian manifolds; see \cite{Atiyah76}.  The formulation given here is only a special
case of a more general algebraic formulation of the conjecture; for the precise statement we refer to Conjecture 10.2 in~\cite{LückBook}.
The Atiyah conjecture is known for certain classes of torsion-free groups, but is
open in general. Recently A.~Jaikin-Zapirain
obtained significant new results concerning the general conjecture; see \cite{Jaikin18}.

\medskip

We will study the $p$-adic Atiyah question in Section~\ref{sec:p-adic-atiyah}. As for the Atiyah conjecture
a purely algebraic formulation will play an important role. We shall prove the following results.
\begin{theorem}\label{thm:p-adic-atiyah-abelian-1}
  Let $X$ be a connected CW-complex with finite $(j+1)$-skeleton.
  Let $(\phi,G)$ be a virtual pro-$p$ completion of $\Gamma = \pi_1(X)$.
  If $G$ is virtually abelian,
  then
  $$b_j^{[p]}(X;\phi,k) \in \bbZ.$$
\end{theorem}
\begin{theorem}\label{thm:p-adic-atiyah-wedge}
  Let $X = X_1 \vee X_2 \vee \dots \vee X_n$ be a CW-complex, which is
  a wedge sum of finite connected CW-complexes $X_i$ with virtually abelian fundamental groups.
  Then all $p$-adic Betti numbers of $X$ are integers.
\end{theorem}

\subsubsection*{Applications}
Applications to cohomology growth provide another incentive to investigate the $p$-adic Atiyah question.
In fact, if the $p$-adic Betti numbers are integers, then there is a useful dichotomy concerning the
possible growth rates of ordinary Betti numbers in towers of finite sheeted covering spaces.
The sequence of Betti numbers either stabilizes at the $p$-adic Betti number or grows relatively fast.
Our most general result in this direction is
Theorem \ref{thm:dichotomy-Frattini} below.

We illustrate this for the special case of $p$-adic analytic towers.
Let $X$ be a finite connected CW-complex $X$ with $\pi_1(X) = \Gamma$.
For every homomorphism $\phi \colon \Gamma \to \GL_m(\bbZ_p)$ we obtain
an associated chain of principal congruence subgroups
$$\Gamma_n = \phi^{-1}\Bigl( \ker\bigl(\GL_m(\bbZ_p) \to \GL_m(\bbZ/p^n\bbZ)\bigr) \Bigr).$$
Note that the closure $G = \overline{\phi(\Gamma)}$ of the image of $\phi$ is a $p$-adic Lie group.
\begin{theorem}\label{thm:p-adic-analytic-towers}
  For a chain $\Gamma_n$ as above, assume that
  $b_j^{[p]}(X;\phi,k) \in \bbZ$. Then
  either the sequence $b_j(\widetilde{X}/\Gamma_n;k)$ stabilizes at $b_j^{[p]}(X;\phi,k)$
  or
  $$ b_j(\widetilde{X}/\Gamma_n;k) \geq \kappa |\Gamma:\Gamma_n|^{1/d} + b_j^{[p]}(X;\phi,k)$$
  for all sufficiently large $n$ and some constant $\kappa > 0$ where
   $d = \dim(G)$.
\end{theorem}

Once again there is an $L^2$-analogue which reveals the antagonistic behaviour of the $p$-adic invariants.
Assume for simplicity that $\phi\colon \Gamma \to \GL_m(\bbZ_p)$ is injective, then
a theorem of Bergeron, Linnell, L\"uck and Sauer (see Thm.~1 in \cite{BLLS14})
yields 
  $$ b_j(\widetilde{X}/\Gamma_n;\bbQ) = b_j^{(2)}(X) |\Gamma:\Gamma_n| + O(|\Gamma:\Gamma_n|^{1-1/d})$$
  as $n$ tends to infinity where $d = \dim(G)$.  Their result is based on Iwasawa theory and a
  similar result holds for homology with $\bbF_p$-coefficients.  Remarkably, their methods do
  not apply to homology with coefficients in $\bbF_\ell$ for primes $\ell \neq p$, whereas our
  methods fail for $k = \bbF_p$.  Similar Iwasawa methods have been applied by Calegari and
  Emerton to prove asymptotic upper bounds for multiplicities of cohomological automorphic
  representations; see \cite{CalegariEmerton09}.  We hope that a better understanding of the
  $p$-adic Atiyah question provides corresponding asymptotic lower bounds.

   \subsection{Outlook and open problems}
  The $p$-adic invariants appear here for the first time and there are several loose ends and
  open questions.  Here is a short list of issues, beyond the $p$-adic Atiyah question, that
  might deserve attention in future research.
  
  \subsubsection{More examples}
  The list of examples for which we can actually compute $p$-adic invariants is quite small.
  At the moment, our main method for computing $p$-adic invariants is the Approximation Theorem
  which, however, seems difficult to apply in more complicated examples.
  For instance, it would be highly interesting to study the $p$-adic Betti numbers of hyperbolic
  three-manifolds. Since closed hyperbolic three-manifolds virtually fibre over the circle \cite{Agol2013},
  this amounts to an investigation of mapping tori.
  
  \subsubsection{}
    \emph{Is there a connected finite aspherical CW-complex $X$ and two injective
  completions $(\phi_1,G_1)$ and $(\phi_2, G_2)$ of its fundamental group such that
  $b_j^{[p]}(X;\phi_1,k) \neq b_j^{[p]}(X;\phi_2,k)$ for some $j$?}
  The $p$-adic Betti numbers of $X$ with respect to two different virtual pro-$p$ completions
  $\phi_1$ and $\phi_2$ are distinct in general.  It is easy to find such examples if
  $\ker(\phi_1) \neq \ker(\phi_2)$; see \S\ref{par:dep-completion-1}.  Even if
  $\phi_1$ and $\phi_2$ are both injective, the $p$-adic Betti numbers can be distinct;
  see \S\ref{par:dep-completion-2}. However, we are not aware of an aspherical example.

  \subsubsection{Formulation in terms of $p$-adic operator algebras}
  Our construction of $p$-adic Betti numbers uses only a small amount of $p$-adic analysis and
  does not involve a $p$-adic companion of the group von Neumann algebra, which plays the
  central role in the $L^2$-theory. Is there a definition of the $p$-adic Betti numbers in
  terms of $p$-adic operator algebras?

  \subsubsection{Generalizing $p$-adic approximation}
  L\"uck's approximation theorem has been generalized in several directions \cite{ABBGNRS,
    BergeronGaboriau2004, Farber1998, Kionke18} and it would be interesting to know whether
  there are $p$-adic analogues of these generalizations.  For instance, is there a $p$-adic
  Farber condition which implies an approximation result like Theorem
  \ref{thm:intro-approximation} for chains of open subgroups
  $G \supseteq G_1 \supseteq G_2 \supseteq \dots$ which are not necessarily normal in~$G$?

  \subsection*{Organization of the article}
  \subsubsection*{Section \ref{sec:preliminaries}: Preliminaries and notation.}
  We define and review basic notions. We recall the notion of free profinite $\bbZ_p$-modules
  and define smooth admissible representations of profinite groups.
  
  \subsubsection*{Section \ref{sec:p-adic-dimension}: The $p$-adic dimension function.}
  In this section we define the K-group $\admK(\cga{k}{G})$ of smooth admissible
  $k$-representations of a pro-$p$ group $G$.  We construct the $p$-adic dimension function
  $$\pdim_k^G\colon \admK(\cga{k}{G}) \to \bbZ_p$$ which plays the key role in defining the
  $p$-adic Betti numbers.  We study the behaviour of the $p$-adic dimension with
  respect to restriction, inflation and induction of representations.

  \subsubsection*{Section \ref{sec:p-adic-card}: The $p$-adic cardinality function.}
  We introduce the admissible Burnside ring $\admA(G)$ of a virtually pro-$p$ group $G$.
  We construct a $p$-adic cardinality function
  $\pcard^G\colon \admA(G) \to \bbZ_p$ which will be used to define the $p$-adic torsion.
  We study the effect of restriction and inflation in the
  admissible Burnside ring.

  \subsubsection*{Section \ref{sec:the-invariants}:  $p$-adic topological invariants.}
  This is the main part of the article.
  We define the $p$-adic Betti numbers and the $p$-adic torsion and we verify that these are homotopy invariants.
  The main result is the Approximation Theorem \ref{thm:approximation}.
  It will be used to establish
  a K\"unneth formula, Poincar\'e duality, a formula for wedge sums and formula for
  the $p$-adic Euler characteristic.

  \subsubsection*{Section \ref{sec:examples}: Examples.}
  Here we discuss a number of examples: tori, surfaces, free groups and the infinite cyclic
  coverings of knot complements.  We illustrate that the $p$-adic Betti numbers depend on the chosen
  completion and we study the $p$-adic torsion of (free abelian)-by-cyclic groups to show
  that the torsion can take transcendental values.
  
  \subsubsection*{Section \ref{sec:p-adic-atiyah}: The $p$-adic Atiyah question.}
  In the last section we discuss the $p$-adic Atiyah question (as explained in the
  introduction) and we explain the applications concerning the growth of Betti numbers.

  \subsection*{Acknowledgements}
I would like to thank Anton Clau{\ss}nitzer and Holger Kammeyer.
Anton was so kind to explain some results of his Ph.D.\ thesis to me.
Holger patiently answered my questions about knots.

\section{Preliminaries and Notation}\label{sec:preliminaries}

\subsection{}
The symbols $\bbN, \bbN_0, \bbZ$ and $\bbQ$ denote the set of natural numbers, the set of non-negative integers, the ring of integers and the field of rational numbers respectively.
For a prime number $p$ the ring of $p$-adic integers is denoted by $\bbZ_p$ and the field
of $p$-adic numbers by $\bbQ_p$. The $p$-adic absolute value of $x \in \bbQ_p$ is  $|x|_p = p^{-\nu_p(x)}$
where $\nu_p$ denotes the $p$-adic valuation.
All rings considered below are assumed to be associative and to have a unit.

\subsection{Profinite groups}
A profinite group $G$ is a compact totally disconnected Hausdorff topological group. Every profinite group
has a neighborhood base of open normal subgroups and is isomorphic to an inverse limit of finite groups.
A profinite group $G$ is a \emph{pro-$p$ group} if it is isomorphic to an inverse limit of finite $p$-groups.
A profinite group is called \emph{finitely generated} if it has a finitely generated dense subgroup.

Let $G$ be a profinite group and let $W$ be a set.
The space of locally constant $W$-valued functions on $G$ will be denoted by 
$C^\infty(G,W)$. This notation will also be used for a finite groups $G$.

\begin{definition}
  Let $G$ be a profinite group.
  An \emph{exhaustive chain} $(N_n)_{n \in \bbN}$ is a decreasing sequence $N_1 \supseteq N_2 \supseteq N_3 \dots$
  of closed normal subgroups $N_n \normal_c G$
  such that $\bigcap_{n \in \bbN} N_n = \{1\} $.
  We note that every \emph{finitely generated}
  profinite group has an exhaustive chain consisting of \emph{open} subgroups.
\end{definition}

\subsection{Virtually pro-$p$ groups and the $p$-adic index}\label{par:index-convention}
Let $p$ be a prime number and let $G$ be a profinite group which is virtually a pro-$p$ group, this means, that $G$ has an open pro-$p$ subgroup.
We sometimes need to interpret the order of $|G|$ or, more generally, the index $|G : K|$ of a closed subgroup $K \leq_c G$ as a $p$-adic number.
We define $\Vert G : K \Vert \in \bbZ_p$ as
$$\Vert G:K \Vert = \begin{cases}  |G:K| \quad & \text{ if $K$ is open in $G$}\\
  0  & \text{ otherwise. } \end{cases}$$
We also write $\Vert G \Vert = \Vert G:\{1\}\Vert$,

Observe that whenever $(N_n)_{n \in \bbN}$ is an exhaustive chain of \emph{open} subgroups in $G$, then
$(|G:N_nK|)_{n\in \bbN}$ is a convergent sequence of $p$-adic numbers and
\begin{equation}\label{eq:index-limit}
  \lim_{n \to \infty} |G:N_nK| = \left\Vert G: K \right\Vert \in \bbZ_p.
\end{equation}

\subsection{Profinite $\bbZ_p$-modules and algebras}
Let $p$ be a prime number.
A \emph{profinite $\bbZ_p$-module} $M$ is a compact Hausdorff topological $\bbZ_p$-module
which admits a neighborhood base of $0$ consisting of open submodules. Equivalently, $M$ is isomorphic to
an inverse limit of finite $\bbZ_p$-modules.  A \emph{profinite $\bbZ_p$-algebra} $A$ is a
commutative compact Hausdorff topological $\bbZ_p$-algebra which admits a neighborhood base
of $0$ which consists of open ideals; equivalently, $A$ is an inverse limit of finite
$\bbZ_p$-algebras.

\subsection{Free  profinite $\bbZ_p$-modules}\label{sec:free-profinite-module}
Let $I$ be a countable set and $M$ a profinite $\bbZ_p$-module.
A map $f\colon I \to M$ is called \emph{$0$-convergent}
if for every open submodule  $U\leq_o M$ all but finitely many $i \in I$ satisfy
$f(i) \in U$.

Let $\calM(I,\bbZ_p)$ denote the set of functions from $I$ to $\bbZ_p$.
Equipped with the topology of pointwise convergence $\calM(I,\bbZ_p)$
is a profinite $\bbZ_p$-module.

For convenience we fix the following notation.
Let $F \subseteq I$ be a finite set and let $\epsilon > 0$.
We define $\calU(F,\epsilon)$
to be the set of all $\eta$ such that $|\eta(i)|_p < \epsilon$
for all $i \in F$.
The set $\calU(F,\epsilon)$ is an open submodule of $\calM(I,\bbZ_p)$ and the collection of these
sets is a neighborhood base for $0$ in $\calM(I,\bbZ_p)$.

Write $\delta_i \in \calM(I,\bbZ_p)$ for the indicator function of $i \in I$. In fact,
$\calM(I,\bbZ_p)$ is \emph{the free profinite $\bbZ_p$-module} over $I$.  In other words,
$\calM(I,\bbZ_p)$ has the following universal property.
\begin{lemma*}
  For every profinite $\bbZ_p$-module $M$ and every $0$-convergent map $f\colon I \to M$
  there exists a unique homomorphism $\widetilde{f} \colon \calM(I,\bbZ_p) \to M$ of profinite $\bbZ_p$-modules
  such that $\widetilde{f}(\delta_i) =  f(i)$ for all $i \in I$.
\end{lemma*}
\begin{proof}
  The $\bbZ_p$-submodule generated by $\{\delta_i \mid i \in I\}$ is the space of finitely
  supported functions on $I$. Since this submodule  is dense in $\calM(I,\bbZ_p)$ in the topology of pointwise
  convergence, the uniqueness is obvious.
  
  In order to construct $\widetilde{f}$ we define, for every open submodule $U \leq_o M$,
  $$\widetilde{f}_U(\alpha) = \sum_{\substack{i \in I \\ f(i) \not\in U} } \alpha(i) f(i) + U \in M/U$$
  for all $\alpha \in \calM(I,\bbZ_p)$.
  The maps $\widetilde{f}_U$ form a compatible family of linear maps into the finite quotients
  of $M$. Since $M \cong \varprojlim_{U \leq_o M} M/U$ we can use the universal property of the inverse limit
  to obtain the desired map $\widetilde{f}$.
\end{proof}

\subsection{Smooth representations of profinite groups}\label{par:smooth-representations}
  Let $R$ be a commutative ring and let
  $G$ be a profinite group. A \emph{smooth representation $(\rho,V)$} (or simply $V$)  of $G$ over $R$ is a
  representation $\rho\colon G \to \Aut_R(V)$ of $G$ on an $R$-module $V$ such that
  the stabilizer of every element $v\in V$ is open in $G$. 
  We say that $V$ is \emph{admissible}, if for every open subgroup $H \subseteq_o G$
  the space $V^H$ of $H$-invariant vectors is finitely generated.
  
  For example, the $R$-module of locally constant functions $C^\infty(G,R)$ carries two
  commuting smooth admissible representations of $G$.  The \emph{left regular representation}
  $l$ of $G$ on $C^\infty(G,R)$ is defined as
  $$\up{{l(g)}}f(x) = f(g^{-1}x)$$
  for all $g,x \in G$ and all $f \in C^\infty(G,R)$,
  The \emph{right regular representation} $r$
  is defined by
  $$\up{{r(g)}}f(x) = f(xg)$$
  for all $g,x \in G$ and all $f \in C^\infty(G,R)$.

\subsection{Induction of smooth representations}\label{par:induction}
Let $G$ be a profinite group and let $H \leq_c G$ be a closed subgroup.
Let $(\sigma, W)$ be a smooth representation of $H$ over $R$.
The representation \emph{(smoothly) induced} by $\sigma$ from $H$ to $G$
is the representation of $G$ on the space
$$\Ind_H^G(W) = \{f \in C^\infty(G,W) \mid (\forall h \in H, x \in G) \; f(xh) = \sigma(h^{-1}) f(x) \}$$
defined by the left regular action
$$\up{{l(g)}}f(x) = f(g^{-1}x).$$
Since every locally constant function $f \in C^\infty(G,W)$
is invariant under some open normal subgroup $N \normal_o G$,
it follows that $\Ind_H^G(W)$ is smooth.

\begin{lemma*}
  (a) Let $N \normal_o G$ be an open normal subgroup.
  Then the space of $N$-invariants in $\Ind_H^G(W)$
  is canonically isomorphic to $\Ind_{HN/N}^{G/N}(W^{H\cap N})$
  where $W^{H \cap N}$ denotes the representation of $HN/N \cong H/(H \cap N)$ on the space
   of $(H\cap N)$-invariants in $W$.

  (b) If $(\sigma, W)$ is admissible, then so is $\Ind_H^G(W)$.

  (c) For every exhaustive chain $(N_n)_{n\in \bbN}$ of open normal subgroups in $G$
  we have
  \[\Ind_H^G(W)  = \varinjlim_{n \in \bbN} \Bigl(\Ind_{HN_n/N_n}^{G/N_n}(W^{H\cap N_n}) \Bigr).\]
\end{lemma*}
\begin{proof}
  Assertion (a) follows from the observation that for every $N$-invariant
  $f \in \Ind_H^G(W)$ the equality
  $$f(x) = f(xn) = \sigma(n^{-1}) f(x)$$
  holds for all $n \in H\cap N$ and all $x \in G$.

  Suppose that $(\sigma,W)$ is admissible,
  then $W^{H\cap N}$ is a finitely generated $R$-module for every $N \normal_o G$.  By (a) the
  $R$-module $\Ind_H^G(W)^N$ is isomorphic to a direct sum of $|G:HN|$ copies of
  $W^{N \cap H}$ and is therefore finitely generated.  Let $(N_n)_{n \in \bbN}$ be an exhaustive
  chain of open normal subgroups in $G$.  Then every locally constant function
  $f \in C^\infty(G,W)$ is invariant under $N_n$ for some $n \in \bbN$.  In particular, we see
  that
\[ \Ind_H^G(W) = \varinjlim_{n \in \bbN} \Ind_H^G(W)^{N_n}  \cong \varinjlim_{n \in \bbN} \Bigl(\Ind_{HN_n/N_n}^{G/N_n}(W^{H\cap N_n}) \Bigr). \]
\end{proof}

\section{The $p$-adic dimension function}\label{sec:p-adic-dimension}

\subsection{}
Fix a prime number $p$.
Throughout the letter $k$ will denote either $\bbQ$ or a finite field of characteristic $\chr(k) = \ell$ with
$\ell \neq p$. The purpose of this section is to define a $p$-adic-valued dimension function
for smooth admissible representations of finitely generated virtually pro-$p$ groups.

\subsection{}
The set of isomorphism classes of smooth irreducible $k$-representations of a profinite group
$G$ will be denoted by $\Irr_k(G)$.  Every smooth irreducible $k$-representation is finite
dimensional and the action of $G$ factors over some finite continuous quotient.
The class of $S$ will be denoted by $[S]$.

\subsection{}
Let $G$ be a pro-$p$ group. Recall that $p$ does not divide the characteristic of $k$.
Every smooth admissble $k$-representation $V$ of $G$ decomposes as a direct sum 
$$ V \cong \bigoplus_{[S] \in \Irr_k(G)} m(V,S) S $$
where each $[S] \in \Irr_k(G)$ occurs with a finite multiplicity $m(V,S) \in \bbN_0$.  The
isomorphism type of $V$ is determined uniquely by the multiplicities. In particular, the
isomorphism classes of smooth admissible $k$-representations of $G$ form a set $\adm_k(G)$.
In fact, $\adm_k(G)$ is a monoid using the direct sum of representations.
The class of a representation $V$ in $\adm_k(G)$ will be denoted by $[V]$.

\begin{definition}
  Let $G$ be a pro-$p$ group.
  The \emph{K-group of smooth admissible $k$-representations} $\admK(\cga{k}{G})$
  is the free profinite $\bbZ_p$-module $\calM(\Irr_k(G),\bbZ_p)$ on $\Irr_k(G)$.
\end{definition}

The multiplicity map $m\colon \adm_k(G) \to \admK(\cga{k}{G})$ sends $[V]$ to
$m([V])\in \admK(\cga{k}{G})$, where  $m([V])$ as a function on $\Irr_k(G)$
is $[S] \mapsto m(V,S)$. Since
$m(V \oplus W, S) = m(V,S) + m(W,S)$ it is an injective monoid homomorphism. We use this
injective map to actually identify $\adm_k(G)$ with its image in the admissible K-group. 

\begin{lemma}\label{lem:finite-dim-dense}
  The set of classes of smooth finite dimensional representations of $G$
  is dense in $\admK(\cga{k}{G})$.
\end{lemma}
\begin{proof}
  Let $\eta \in \admK(\cga{k}{G})$ and $F \subseteq \Irr_k(G)$ be a finite set.
  We consider the open neighbourhood $\eta + \calU(F,\epsilon)$ as
  in \S\ref{sec:free-profinite-module}.  For every $[S] \in F$ we find, using that $\bbN_0$ is
  dense in $\bbZ_p$, a number $n_S \in \bbN_0$ such that
  $\left|n_S-\eta([S])\right|_p < \epsilon$.  Hence, the class of the finite dimensional
  smooth $k$-representation
  $$V = \bigoplus_{[S] \in F} n_S S $$ is contained in $\eta + \calU(F,\epsilon)$.
\end{proof}

\begin{theorem}\label{thm:dimension-function}
  Let $p$ be a prime and $G$ a finitely generated pro-$p$ group.
  Recall that $k$ denotes either $\bbQ$ or a finite field with $\chr(k) \neq p$.
  There is a unique continuous $\bbZ_p$-linear map
  $$\pdim^{G}_k\colon \admK(\cga{k}{G}) \to \bbZ_p$$
  such that $\pdim_k^{G} V = \dim_kV$,
  whenever $V$ is finite dimensional over $k$.
\end{theorem}
\begin{proof}
  We use that $\admK(\cga{k}{G})$
  is the free profinite $\bbZ_p$-module over $\Irr_k(G)$.
  Lemma \ref{lem:keylem} below shows that the map 
  $\dim_k\colon [S] \mapsto \dim_k(S)$ is  $0$-convergent.
  Therefore we can
  deduce the existence of a continuous $\bbZ_p$-linear map $\pdim_k^G$
  which satisfies $\pdim_k^G(S) = \dim_kS$ for all $[S] \in \Irr_k(G)$ from \S\ref{sec:free-profinite-module}.
  Since every finite dimensional smooth representation is the direct sum of a finite number
  of irreducible ones, we conclude that $\pdim_k^G$ has the desired property.
\end{proof}

\begin{lemma}\label{lem:keylem}
  Let $G$ be a finitely generated  pro-$p$ group.
  For every $\nu \in \bbN$ all but finitely many
  $[S]\in \Irr_k(G)$ satisfy
  $p^\nu | \dim_k(S)$. 
\end{lemma}
\begin{proof}
  \emph{Step 1}: There are only finitely many conjugacy classes of continuous
  homomorphisms $\psi\colon G \to \GL_n(k)$.

  Suppose that $k$ is a finite field. By assumption $G$ is finitely generated and since
  $\GL_n(k)$ is a finite group, there are at most $\left|\GL_n(k)\right|^{d(G)}$ many homomorphisms $\psi$ as above.

  Now assume that $k = \bbQ$. By a well-known Theorem of Jordan (1880) and Minkowski (1887)
  there are only finitely many conjugacy classes of finite subgroups in $\GL_n(\bbQ)$.  Choose
  representatives $F_1,\dots,F_j$ for the conjugacy classes. As above, for each $i$ there are
  only finitely many continuous homomorphisms from $G$ to $F_i$.

  \medskip

  \emph{Step 2:} Every irreducible smooth
  $\alg{k}$-representation of $G$, where $\alg{k}$ denotes the algebraic closure of $k$,
  has dimension $p^j$ for some $j \in \bbN_0$.

  Since every smooth representation factors over a continuous finite quotient, we may assume
  that $G$ is a finite $p$-group.  Suppose that $k$ is a finite field. Since
  $\chr(k) = \ell \neq p$, the dimensions of irreducible representations of $G$ over $\alg{k}$
  coincide with dimensions of irreducible representations of $G$ over $\alg{\bbQ}$. Indeed,
  this follows since Brauer characters and ordinary characters coincide in this case; see Corollary (18.11)
  in \cite{CurtisReinerI}. Hence it suffices to consider the case $k = \bbQ$.  In this case
  $\alg{k}$ is an algebraically closed field of characteristic zero and the assertion follows
  from Theorem (3.11) in \cite{Isaacs}.

  \medskip

  \emph{Step 3:}
  The dimension of a smooth irreducible $k$-representation of $G$ divides $(p-1)p^j$ for some $j\in\bbN$.

  As before we my assume that $G$ is a finite $p$-group.
  Let $S$ be an irreducible $k$-representation of $G$.
  The $\alg{k}$-representation $\alg{k} \otimes_k S$
  decomposes into $\alg{k}$-irreducible representations
  $$ \alg{k} \otimes_k S = \bigoplus_{i=1}^f m S_i$$
  where each occurs with the same multiplicity $m$ and, moreover, the $S_i$ form one orbit
  under the action of the Galois group $\Gal(\alg{k}/k)$; see (9.21) in \cite{Isaacs}.  Let
  $k(p^\infty)$ be the subfield of $k$ obtained by adjoining all $p$-power roots of unity to
  $k$.  Since $k(p^\infty)$ is a splitting field for $G$, the representations $S_i$ are
  $\Gal(\alg{k}/k(p^\infty))$ invariant. We deduce that
  $f$ is the degree of a finite field extension $k'/k$ with $k' \subseteq k(p^\infty)$
  and thus divides $[k(\zeta):k]$ where $\zeta$ is a $p^{a}$-th root of unity for some $a \in \bbN$.

  Assume that $k = \bbQ$. By a theorem of Roquette $m = 1$ or $m=2$ where the latter can only
  occur if $p = 2$; see (10.14) in \cite{Isaacs}.
  Moreover, $[\bbQ(\zeta):\bbQ] = \phi(p^a) = (p-1)p^{a-1}$ and we deduce with Step 2 that
  $\dim_k(S) = fm\dim_{\alg{k}}(S_1)$ divides $(p-1)p^j$ for some $j$.
  Finally, assume that $k$ is a finite field. In this case $m = 1$ \cite[(9.21)]{Isaacs}
  and $[k(\zeta):k]$ divides $\phi(p^{a})$, thus  we may conclude as above.

  \medskip
  
  Now we can complete the proof. Let $\nu\in \bbN$ be given.  We deduce from Step 1 that for
  each $n \in \bbN$ the number $r_k(n)$ of isomorphism classes of $n$-dimensional smooth
  irreducible $k$-representations of $G$ is finite.  In particular, there is only a finite set
  $X \subseteq \Irr_k(G)$ of representations whose dimension divides $(p-1)p^{\nu-1}$.  Let
  $[S] \in \Irr_k(G) \setminus X$. By Step 3 $\dim_k(S) | (p-1)p^j$ for some $j\in
  \bbN$. However, since $\dim_k(S) \nmid (p-1)p^{\nu-1}$ we conclude that
  $p^\nu \mid \dim_k(S)$.
\end{proof}

\subsection{Restriction to open subgroups}\label{par:res-dim}
Let $H \leq_o G$ be an open subgroup of a finitely generated pro-$p$ group $G$.  The
restriction $V_{|H}$ to $H$ of a smooth admissible $k$-representation $V$ of $G$ is a smooth
admissible representation of $H$. The next result extends this operation
to a continuous map between the admissible K-groups.

\begin{proposition*}
  Let $G$ be a finitely generated pro-$p$ group and let $H \leq_o G$.
  There is a unique continuous
  $\bbZ_p$-linear map $$\Res_H^G \colon \admK(\cga{k}{G}) \to \admK(\cga{k}{H})$$
  such that $\Res_H^G([V]) = [V_{|H}]$ for every
  $[V] \in \adm_k(G)$.
  
  Restriction preserves the $p$-adic dimension, i.e.
  $$\pdim_k^H(\Res_H^G(\eta)) = \pdim_k^G(\eta)$$
  for all $\eta \in \admK(\cga{k}{G})$.
\end{proposition*}
\begin{proof}
  Assume for a moment that $\Res_H^G$ exists, then
  the map $\pdim_k^H \circ \Res_H^G$ is a continuous $\bbZ_p$-linear map
  from $\admK(\cga{k}{G})$ to $\bbZ_p$ and satisfies
  $$ \pdim_k^H(\Res_H^G([V])) = \pdim_k^H([V_{|H}]) = \dim_k V_{|H} = \dim_k V$$
  for every finite dimensional smooth $k$-representation $V$ of $G$.
  The uniqueness statement in Theorem \ref{thm:dimension-function}
  yields $\pdim_k^H \circ \Res_H^G = \pdim_k^G$.
  
  We will construct $\Res_H^G$ using the universal property explained in \S\ref{sec:free-profinite-module}.
  This amounts to showing that the map $\Irr_k(G) \to \admK(\cga{k}{H})$
  which sends $[S]$ to $[S_{|H}]$ is $0$-convergent.
  Let $F \subseteq \Irr_k(H)$ be a finite set and let $\epsilon > 0$.
  Then $[S_{|H}]$ does not lie in the open submodule $\calU(F,\epsilon) \subseteq \admK(\cga{k}{G})$
  only if $m(S_{|H},T) \neq 0$ for some $[T] \in F$.
  
  It remains to show that for every
  $[T] \in \Irr_k(H)$ there are only finitely many $[S] \in \Irr_k(G)$
  such that $m(S_{|H},T) \neq 0$.
  Indeed,  $m(S_{|H},T) \neq 0$ exactly if there is a non-zero $H$-equivariant homomorphism
  from $T$ to $S_{|H}$. By Frobenius reciprocity the later holds if and only if there
  is a non-zero homomorphism of $G$-representations
  from the induced representation $\Ind_H^G(T)$ to $S$; see (10.8) in \cite{CurtisReinerI}.
  However, a non-trivial homomorphism from $\Ind_H^G(T)$ to the irreducible representation $S$ exists precisely
  when $S$ occurs as a direct summand of $\Ind_H^G(T)$.
  Since the representation $\Ind_H^G(T)$ is finite dimensional, its dimension is $|G:H|\dim_k(T)$,
  there are at most finitely many such $[S] \in \Irr_k(G)$.
\end{proof}

\begin{definition}\label{def:p-adic-dim}
  Let $G$ be a finitely generated virtually pro-$p$ group and let $V$ be a smooth admissible $k$-representation
  of $G$.
  We define the \emph{$p$-adic dimension} of $V$ as
  \begin{equation}\label{eq:dim-restriction}
    \pdim_k^G V = \pdim_k^H ([V_{|H}])
  \end{equation}
  
  where $H \leq_o G$ is some open pro-$p$ subgroup.  The previous proposition shows that
  this definition is independent of the choice of the subgroup $H$.
  In addition, with this definition any open subgroup $H \leq_0 G$
  satisfies the identity \eqref{eq:dim-restriction}.
\end{definition}

\subsection{Inflation}\label{sec:inflation}
Let $\pi \colon G \to H$ be a surjective homomorphism of finitely generated pro-$p$ groups.
Then every smooth admissble $k$-representation $V$ of $H$ becomes a smooth admissible
$k$-representation of $G$ via $\pi$. This new representation $\infl_H^G(V)$ is called the \emph{inflation} of
$V$ via $\pi$.
\begin{proposition*}
  There is a unique continuous $\bbZ_p$-linear map
  $\infl_H^G \colon \admK(\cga{k}{H}) \to \admK(\cga{k}{G})$
  such that $\infl_H^G[V] = [\infl_H^G(V)]$ for every
  smooth admissible $k$-representation of $H$.
  Inflation preserves the p-adic dimension, i.e.
  \begin{equation}\label{eq:inflation-preserves-pdim}
    \pdim_k^G(\infl_H^G(\eta)) = \pdim_k^H(\eta).
  \end{equation}
\end{proposition*}
\begin{proof}
  Again uniqueness and \eqref{eq:inflation-preserves-pdim} follow immediately once the
  existence is established.  Let $[S] \in \Irr_k(H)$, then the inflation $\infl_H^G(S)$ is
  irreducible.  This yields an injective map from $\Irr_k(H)$ into $\Irr_k(G)$.  In addition,
  a class $[T] \in \Irr_k(G)$ can be represented by an inflated representation if and only if
  $T$ factors over the kernel of $\pi$.  Hence, for $\eta \in \admK(\cga{k}{H})$ we define
  $$\infl_H^G(\eta)([T]) = \eta([S])$$
  if $T \cong \infl_H^G(S)$ for some $S$ and we impose $\infl_H^G(\eta)([T]) = 0$ otherwise.
  It is easily verified that this defines a $0$-convergent map a $\Irr_k(H) \to \admK(\cga{k}{G})$.
\end{proof}

\begin{lemma}[Approximation Lemma]\label{lem:approximation-dim}
    Let $G$ be a finitely generated virtually pro-$p$ group and let $V$ be a smooth admissible
    $k$-representation.
    Every exhaustive chain $(N_n)_{n\in \bbN}$  in $G$ satisfies
    \[ \lim_{n \to \infty} \pdim_k^{G/N_n} V^{N_n} = \pdim_k^{G} V. \]

    In particular, if each $N_n$ is open in $G$, then
    \[\lim_{n \to \infty} \dim_k V^{N_n} = \pdim_k^G V \in \bbZ_p.\]
  \end{lemma}
  \begin{proof}
    Let $H\normal_o G$ be an open normal subgroup.
    We claim that there is some $e \in \bbN$ such that $N_n \subseteq H$ for all $n \geq e$.
    Indeed, the intersection $\bigcap_{n \in \bbN} (N_n\setminus H) = \emptyset$
    and, since the sets $N_n \setminus H$ are compact and nested, this implies that $N_e \setminus H$ is empty
    for some $e \in \bbN$. Applying this to an open normal pro-$p$ subgroup
    reduces the proof to the case where $G$ is a pro-$p$ group.

    Due to \ref{sec:inflation} we have the equality
    $\pdim_k^G \infl_{G/N_n}^G(V^{N_n}) = \pdim_k^{G/N_n}V^{N_n}$.
    It remains to show that the sequence
    $([\infl_{G/N_n}^G(V^{N_n})])_{n \in \bbN}$ converges to $[V]$ in $\admK(\cga{k}{G})$, then
    the continuity of the $p$-adic dimension function completes the proof.
    Let $A \subseteq \Irr_k(G)$ be a finite subset,
    then there is an open normal subgroup $H \normal_o G$
    such that $H$ acts trivially on $S$ for all $[S] \in A$.
    We have seen that $N_n \subseteq H$ for all sufficiently large $n \in \bbN$, and thus
    $m(V,S) = m(\infl_{G/N_n}^G(V^{N_n}), S)$ for all $[S] \in A$.
    This finishes the proof.
  \end{proof}

  \begin{proposition}\label{prop:induction-dim}
    Let $H \leq_c G$ be a finitely generated closed subgroup of a finitely generated
    virtually pro-$p$ group $G$.
  The p-adic dimension satisfies
  \begin{equation*}
    \pdim_k^G(\Ind_H^G(W)) = \left\Vert G:H \right\Vert \pdim_k^H(W).
  \end{equation*}
  for every smooth admissible $k$-representation $W$ of $H$. Here we use the notation of
  \S\ref{par:index-convention}.
\end{proposition}
\begin{proof}
  Let $(N_n)_{n\in \bbN}$ be an exhaustive chain of open subgroups in $G$.
  We observe that $(H \cap N_n)_{n\in\bbN}$ is an exhaustive chain in $H$.
  We combine 
  Lemma \ref{par:induction} (a),
  the Approximation Lemma \ref{lem:approximation-dim} and equation \eqref{eq:index-limit} to obtain
  \begin{align*}
  \pdim_k^G(\Ind_H^G(W)) &= \lim_{n \to \infty} \pdim_k^{G/N_n} \Ind^{G/N_n}_{HN_n/N_n}(W^{H\cap N_n})\\
    &=  \lim_{n \to \infty} |G:HN_n| \dim_k(W^{H\cap N_n}) = \left\Vert G:H \right\Vert \pdim^H_k(W)
  \end{align*}
\end{proof}

  \section{The $p$-adic cardinality function}\label{sec:p-adic-card}
  \subsection{} Throughout $p$ denotes a fixed prime number and $G$ a finitely generated
  virtually pro-$p$ group. We will define the notion of admissible $G$-set and define a
  corresponding admissible Burnside ring. On this ring there is a continuous $p$-adic valued notion of cardinality.

  \begin{definition}\label{def:admissible-G-set}
    A set $X$ with a $G$-action is \emph{admissible}
    if every $x \in X$ has an open stabilizer in $G$ and for every open subgroup $H \leq_o G$
    the set of $H$-invariants $X^H$ is finite.
  \end{definition}

  \subsection{}
  Let $X$ be an admissible $G$-set. It follows from the definition of admissibility that every
  $G$-orbit in $X$ is finite and, moreover, for every open subgroup $H\leq_o G$ there are only
  finitely many orbits isomorphic to $G/H$ contained in $X$.  Let $\Or(G)$ denote the set of
  isomorphism classes of  finite transitive $G$-sets; it is in bijection with the set of conjugacy
  classes of open subgroups of $G$. Now we can write
  $$ X \cong \bigsqcup_{ Z \in \Or(G)} \iota(X,Z)\:  Z  $$
  as a disjoint union of finite orbits, where each orbit type $Z$ occurs with a finite
  multiplicity $\iota(X,Z) \in \bbN_0$.
  Clearly, the isomorphism type of $X$ is uniquely determined by the multiplicities $\iota(X,Z)$.
  Finite disjoint unions and finite direct products of admissible $G$-sets are admissible.

  \subsection{The admissible Burnside ring}
  Let $\admA(G) = \calM(\Or(G),\bbZ_p)$ denote the free profinite $\bbZ_p$-module on $\Or(G)$;
  compare \S\ref{sec:free-profinite-module}.  Here we should not think of the elements of $\admA(G)$ as
  functions; instead it is  more convenient to denote $\alpha \in \admA(G)$ as a
  formal series
  $$ \alpha = \sum_{Z \in \Or(G)} \alpha_Z Z$$
  with $\bbZ_p$-coefficients.
  For $\alpha, \beta \in \admA(G)$ we define
  $$\alpha\cdot\beta = \sum_{Z \in \Or(G)} \Bigl(\sum_{X, Y\in \Or(G)} \iota(X\times Y,Z) \alpha_X \beta_Y\Bigr) Z $$
  Note that for each $Z \in \Or(G)$ there are only finitely many $X, Y \in \Or(G)$ such that $X\times Y$ contains an orbit
  isomorphic to $Z$.  This defines an associative, commutative multiplication on $\admA(G)$
  which becomes thus a profinite $\bbZ_p$-algebra.

  We will call $\admA(G)$ the \emph{admissible Burnside ring} of $G$. Let $X$ be an admissible $G$-set,
  then $\iota(X) \in \admA(G)$ will denote the element
  $$ \iota(X) = \sum_{Z \in \Or(G)} \iota(X,Z) Z.$$
  We note that $\iota(X) + \iota(Y) = \iota( X \sqcup Y)$ and $\iota(X) \cdot \iota(Y) = \iota(X\times Y)$.
  \begin{lemma}[Universal property]\label{lem:universal-property-burnside-ring}
    Let $A$ be a profinite $\bbZ_p$-algebra.
    Let $f$ be a function which assigns to every finite $G$-set $X$ an element $f(X) \in A$
    in such a way that
    \begin{enumerate}[  (i)]
    \item\label{it:zero-convergent-ass} $f_{|\Or(G)}$ is $0$-convergent
    \item\label{it:additive-ass} $f(X \sqcup Y) = f(X) + f(Y)$ and
    \item\label{it:multiplicative-ass} $f(X \times Y) = f(X)  f(Y)$ for all finite $G$-sets $X$ and $Y$.
    \end{enumerate}
    Then there is a unique homomorphism $\widetilde{f}\colon \admA(G)\to A$ of profinite $\bbZ_p$-algebras
    with $\widetilde{f}(\iota(X)) = f(X)$ for all finite $G$-sets.
  \end{lemma}
  \begin{proof}
    Since $\admA(G)$ is the free profinite $\bbZ_p$-module on $\Or(G)$
    assumptions \eqref{it:zero-convergent-ass} and \eqref{it:additive-ass}
    imply the existence of a unique homomorphism $\widetilde{f} \colon \admA(G) \to A$
    of profinite $\bbZ_p$-modules with $\widetilde{f} \circ \iota = f$.
    It remains to show that $\widetilde{f}$ is multiplicative.
    We observe that, since $\bbN_0$ is dense in $\bbZ_p$, the image under $\iota$ of the set of finite
    $G$-sets is dense in $\admA(G)$;
    compare with the proof of Lemma~\ref{lem:finite-dim-dense}.
    Now continuity of $\widetilde{f}$ and assumption \eqref{it:multiplicative-ass} yield the claim.
  \end{proof}
  
  \begin{theorem}
    Let $G$ be a finitely generated virtually pro-$p$ group.
    There is a unique continuous homomorphism
    $\pcard^G \colon \admA(G) \to \bbZ_p$
    of $\bbZ_p$-algebras
    such that $\pcard^G(\iota(X)) = |X|$ for every finite $G$-set $X$.
  \end{theorem}
  \begin{proof}
    We use the universal property of the admissible Burnside ring. The function $|\cdot|$ is
    additive and multiplicative, therefore we only need to show that
    $|\cdot| \colon \Or(G) \to \bbZ_p$ is
    $0$-convergent. Indeed, fix some open normal pro-$p$ subgroup $H \normal_o G$.  The
    cardinality $|Z|$ of an orbit type is the index of the stabilizer $K \leq_o G$ of some
    point in $Z$.  Since $G$ is finitely generated it has only finitely many open subgroups of a given index.
    Further, $|G:K|$ is of the form $dp^\nu$ for some $\nu \in \bbN_0$ and $d$
    a divisor of the index $|G:H|$. We deduce that the map is $0$-convergent and this completes the proof.
  \end{proof}

  \begin{definition}
    Let $X$ be an admissible $G$-set,
    then $\pcard^G(X) := \pcard^G(\iota(X))$ will be called the \emph{$p$-adic cardinality}
    of $X$.
  \end{definition}

  \subsection{Restriction to open subgroups}\label{par:res-card}
  Let $G$ be a finitely generated virtually pro-$p$ group and $H\leq_o G$ be an open subgroup.
  Given an admissible $G$-set $X$, we may restrict the action to $H$ in order to obtain
  an admissible $H$-set $X_{|H}$.
  \begin{proposition*}
    There is a unique homomorphism of profinite $\bbZ_p$-algebras
    $$\Res_H^G \colon \admA(G) \to \admA(H)$$ such that $\Res_H^G(\iota(X)) = \iota(X_{|H})$ for
    every admissible $G$-set $X$.
    Moreover, $\pcard^G(\alpha) = \pcard^H(\Res_H^G(\alpha))$ for all $\alpha \in \admA(G)$.
  \end{proposition*}
  \begin{proof}
    Existence and uniqueness of the restriction homomorphism follow from the universal
    property of the admissible Burnside ring.  It is clear that restriction is additive and
    multiplicative, thus it remains to show that $\Res_H^G \colon \Or(G) \to \admA(H)$ is
    $0$-convergent.  This follows from the observation that for each fixed $Y \in \Or(H)$
    there are only finitely many $Z \in \Or(G)$ such that $Y$ occurs as an $H$-orbit in
    $Z_{|H}$.

    The equality $\pcard^G(\alpha) = \pcard^H(\Res_H^G(\alpha))$ is a direct consequence of
    the uniqueness of cardinality and restriction homomorphisms.
  \end{proof}
  
  \subsection{Inflation}
  Let $\pi \colon G \to H$ be a surjective homomorphism of finitely generated virtually
  pro-$p$ groups.  Every admissible $H$-set $Y$ of $H$ becomes an admissible $G$-set via
  $\pi$; it will be denoted by $\infl_H^G(Y)$ and is called the \emph{inflation} of $Y$ via
  $\pi$.
  \begin{proposition*}
    There is a unique homomorphism of profinite $\bbZ_p$-algebra
    $$ \infl_H^G \colon \admA(H) \to \admA(G)$$
    such that $\infl_H^G(\iota(Y)) = \iota(\infl_H^G(Y))$ for every admissible $H$-set $Y$.
    Furthermore, $\pcard^G(\infl_H^G(\alpha)) = \pcard^H(\alpha)$ for all
    $\alpha \in \admA(H)$.
  \end{proposition*}
  \begin{proof}
    As before this is an immediate consequence of the universal property of the admissible Burnside ring.
  \end{proof}

  \begin{lemma}[Approximation Lemma]\label{lem:approximation-card}
    Let $G$ be a finitely generated virtually pro-$p$ group and let $X$ be an admissible $G$-set.
    Every exhaustive chain $(N_n)_{n \in \bbN}$ in $G$ 
     satisfies
    \[ \lim_{n \to \infty} \pcard^{G/N_n} X^{N_n} = \pcard^G X.\]
    In particular, if every $N_n$ is open in $G$, then
    \[\lim_{n \to \infty} |X^{N_n}| = \pcard^G X.\]
  \end{lemma}
  \begin{proof}
    We proceed as in the proof of Lemma \ref{lem:approximation-dim}.
    Since $\pcard^G(\infl_{G/N_n}^G(X^{N_n})) = \pcard^{G/N_n} X^{N_n}$
    and the $p$-adic cardinality is continuous, it is sufficient to check
    that the sequence $\infl_{G/N_n}^G(X^{N_n})$ converges to $X$ in $\admA(G)$.
    However, for every orbit type $Z \in \Or(G)$
    the multiplicity $\iota(\infl_{G/N_n}^G(X^{N_n}),Z)$ equals $\iota(X,Z)$
    as soon as the action of $G$ on $Z$ factors through $G/N_n$. 
  \end{proof}

  \section{$p$-adic topological invariants}\label{sec:the-invariants}

  Let $X$ be a connected, locally path-connected and semilocally simply-connected topological
  space. The fundamental group of $X$ at some fixed basepoint $x_0\in X$ will be denoted by
  $\Gamma = \pi_1(X,x_0)$.  We write $q\colon \widetilde{X} \to X$ to denote the universal covering space of
  $X$. The choice of a base point $\tilde{x}_0 \in \widetilde{X}$ with $q(\tilde{x}_0) = x_0$ yields
  an isomorphism between $\Gamma$ and the group of deck transformations of $\widetilde{X}$.
  In this way the universal covering carries a left action of $\Gamma$  such that
  $\widetilde{X}/\Gamma \cong X$.  Throughout we fix a prime number $p$. The purpose of this
  section is to construct $p$-adic-valued invariants of $X$ for every virtual pro-$p$
  completion of $\Gamma$.

  \begin{definition}\label{def:v-pro-p-completion}
    A \emph{virtual pro-$p$ completion $(\phi, G)$} of $\Gamma$ is
    a homomorphism $\phi \colon \Gamma \to G$
    with dense image to a \emph{finitely generated virtually pro-$p$ group} $G$.
  \end{definition}

  \subsection{Direct limits of cohomology groups}\label{sec:direct-limit}
  Let $(\phi,G)$ be a virtual pro-$p$ completion of $\Gamma$ and let $R$ be a commutative
  ring.

  The inverse image $\Gamma_K = \phi^{-1}(K)$ of an open subgroup $K \leq_o G$ is a finite index subgroup
  of $\Gamma$.  Moreover, for open subgroups $K_1 \subseteq K_2 \leq_o G$ we obtain a
  corresponding finite sheeted covering map
  $$q_{K_1,K_2} \colon \widetilde{X}/\Gamma_{K_1} \to \widetilde{X}/\Gamma_{K_2}.$$ For
  every $j \in \bbN_0$ the maps $q_{N_1,N_2}^*$ provide a directed system of cohomology $R$-modules
  $(H^j(\widetilde{X}/\Gamma_N,R))_{N \normal_o G}$ indexed by
  the open normal subgroups of~$G$.  The direct limit will be denoted by
  $$\bar{H}^j(X;\phi,R) = \varinjlim_{N\normal_o G} H^j(\widetilde{X}/\Gamma_N;R).$$
  More generally, let $A \subseteq X$ be a subset.
  For brevity we put $A_K = q_{K,G}^{-1}(A) \subseteq \widetilde{X}/\Gamma_K$ for any
  open subgroup $K \leq_o G$.
  We define
  $$\bar{H}^j(X,A;\phi,R) =  \varinjlim_{N\normal_o G} H^j(\widetilde{X}/\Gamma_N, A_N; R).$$
  The canonical map
  $H^j(\widetilde{X}/\Gamma_N, A_N;R) \to \bar{H}^j(X, A;\phi,R)$ will be denoted
  by $i_N$.  The group $\Gamma/\Gamma_N$ acts $R$-linearly on
  $H^j(\widetilde{X}/\Gamma_N, A_N;R)$.  Since moreover, by assumption, the image
  of $\phi$ is dense, the homomorphism $\phi$ yields an isomorphism
  $\Gamma/\Gamma_N \cong G/N$. Using this isomorphism we obtain an $R$-linear representation of $G/N$
  on $H^j(\widetilde{X}/\Gamma_N, A_N;R)$.  In the limit this provides an
  $R$-linear representation of $G$ on $\bar{H}^j(X,A; \phi,R)$. The action of $G$ is smooth,
  i.e.\ the stabilizer of every element is open in $G$.

  \begin{lemma}\label{lem:exact-invariants}
    Let $A \subseteq X$ be a subspace. Let $R$ be a commutative ring such 
    that $p$ is invertible in~$R$.

    (a) Let $K \leq_oG$ be an open pro-$p$ subgroup.  The
    space $\bar{H}^j(X,A;\phi,R)^K$ of $K$-invariants is canonically isomorphic
    to $H^j(\widetilde{X}/\Gamma_K, A_K ;R)$ as $R$-module.

    (b) Let $K \normal_c G$ be a closed normal pro-$p$ subgroup and let $f \colon G \to G/K$
    denote the factor map.  There is a canonical isomorphism
    $$\bar{H}^j(X,A; f\circ\phi,R) \cong \bar{H}^j(X,A;\phi,R)^K$$
    of $R$-modules with $G/K$-action.
 
  \end{lemma}
  \begin{proof}
     (a) For open normal subgroups $N \normal_o G$ with
    $N \subseteq K$ the maps
    $$q_{N,K}^*\colon H^j(\widetilde{X}/ \Gamma_K,A_K;R) \to H^j(\widetilde{X}/\Gamma_N,A_N;R)$$
    form a compatible system and induce a canonical map $i_K$ from the relative cohomology 
    $H^j(\widetilde{X}/\Gamma_K,A_K;R)$ to $\bar{H}^j(X,A;\phi,R)$ such that
    $i_K = i_N \circ q_{N,K}^*$ for all $N \normal_o G$ with $N \subseteq K$.  Since $K$ is a
    pro-$p$ group and $p$ is invertible in $R$, it follows from the argument used in
    \cite[Prop.~3G.1]{Hatcher} that $q_{N,K}^*$ is injective and the image is exactly the
    space of $K$-invariants in $H^j(\widetilde{X}/\Gamma_N,A_N;R)$.  We deduce that $i_K$ is
    injective, since an element in the kernel lies in the kernel of $q_{N,K}^*$ for some small
    $N \normal_o G$.  Clearly, the classes in the image of $i_K$ are $K$-invariant.
    Conversely, let $ \alpha \in \bar{H}^j(X, A;\phi,R)^K$ be a $K$-invariant class. Then
    $\alpha = i_N(\alpha')$ for some $K$-invariant class
    $\alpha' \in H^j(\widetilde{X}/\Gamma_N, A_N; R)^K$ and a small open normal pro-$p$ subgroup
    $N \normal_o G$.  As observed above, $\alpha'$ lies in the image of $q^*_{N,K}$
    and therefore $\alpha$ lies in the image of $i_N \circ q_{N,K}^*= i_K$.

    (b) The module $\bar{H}^j(X,A; f\circ \phi, R)$ can be constructed as the direct limit of
    the $H^j(\widetilde{X}/\Gamma_N,A_N;R)$ over all open normal subgroups
    $N\normal_o G$ \emph{which contain $K$}. This is a subsystem of the directed system used to
    define $\bar{H}^j(X,A;\phi,R)$ and the inclusion of directed systems yields a canonical map
    $$i_K\colon \bar{H}^j(X,A;f\circ \phi, R) \to \bar{H}^j(X,A; \phi,R)$$
    which satisfies
    $i_K \circ i_{N/K} = i_N$ for all $K \subseteq N \normal_o G$.  The map $i_K$ is
    $G$-equivariant and thus takes values in the subspace of $K$-invariants.  Let
    $\alpha \in \bar{H}^j(X,A;\phi,R)^K$. Since the action of $G$ is smooth, there is some
    $N \normal_o G$ such that $\alpha$ is $N$-stable. We may assume that $N$ contains $K$ and
    that $N$ is pro-$p$. By (a) we know that $\alpha$ lies in the image of $i_N$ and
    $i_K \circ i_{N/K} = i_N$ yields that $\alpha$ lies in the image of $i_K$.  As in (a) the
    map $i_K$ is injective since all the maps $q^*_{N_1,N_2}$ are injective; it provides the
    required isomorphism onto the space of $K$-invariants.
  \end{proof}

    \subsection{Homotopy invariance}\label{par:homotopy-invariance}
    Let $f \colon X \to Y$ be a homotopy equivalence
    of connected, locally path-connected and semilocally simply-connected spaces.
    We choose base points $x_0 \in X$ and $y_0 \in Y$ in such a way that
    $f(x_0) = y_0$. The fundamental groups will be denoted by $\Gamma=\pi_1(X,x_0)$
    and $\Delta = \pi_1(Y,y_0)$.
    
    Let $\phi\colon \Gamma \to G$ and $\psi \colon \Delta \to G$
    be virtual pro-$p$ completions
    such that the following diagram commutes.
    \begin{equation}\label{diag:commuting-completions}
      \begin{tikzcd}
        \Gamma = \pi_1(X,x_0) \arrow[rr,"f_*"] \arrow[dr,"\phi"'] & & \pi_1(Y,y_0) = \Delta \arrow[dl,"\psi"] \\
         & G &
      \end{tikzcd}
    \end{equation}
    \begin{lemma*}
      In the situation above,
      the homotopy equivalence $f$ induces an $R$-linear isomorphism
      of smooth representations of $G$
      $$ \bar{f}^* \colon \bar{H}^j(Y;\psi,R) \to \bar{H}^j(X;\phi,R) $$
      for every commutative ring $R$.
    \end{lemma*}
    \begin{proof}
      Recall that we fixed base points $\tilde{x}_0 \in \widetilde{X}$ and $\tilde{y}_0 \in \widetilde{Y}$
      in order to identify the groups of deck transformations with the fundamental groups.
      Let $\tilde{f} \colon \widetilde{X} \to \widetilde{Y}$ the unique lift of $f$
      such that $\tilde{f}(\tilde{x}_0) = \tilde{y_0}$.
      Then $\widetilde{f}$ is a homotopy equivalence which intertwines the actions of $\Gamma$ and $\Delta$, i.e.
      $$ \tilde{f}(\gamma  x) = f_*(\gamma)  \tilde{f}(x)$$
      for all $x \in \widetilde{X}$ and $\gamma \in \Gamma$.

      For every open normal subgroup $N \normal_o G$, the subgroups $\Gamma_N = \phi^{-1}(N)$
      and $\Delta_N = \phi^{-1}(N)$ satisfy $f_*(\Gamma_N) = \Delta_N$ due to
      \eqref{diag:commuting-completions}. In particular, $\tilde{f}$ induces a homotopy equivalence
      $$ \tilde{f}_N \colon \widetilde{X}/\Gamma_N \to \widetilde{Y}/\Delta_N$$
      which intertwines the actions of $\Gamma/\Gamma_N \cong G/N \cong \Delta/\Delta_N$.
      Thus the map $\tilde{f}_N^*$ induced in the cohomology is an $R$-linear $G/N$-equivariant map.
      The maps  $(\tilde{f}_N^*)_{N\normal_o G}$ are compatible with the directed system and 
      $ \bar{f}^* = \varinjlim_{N \normal_o G} \tilde{f}_N^*$ is the desired isomorphism.
    \end{proof}

    \subsection{The cochain complex}\label{par:cochain-complex}
    Here we give a description
    of a cochain complex which computes $\bar{H}^j(X;\phi,R)$ in the case where $X$ is a connected CW-complex
    with \emph{finite $d$-skeleton} for some $d > j$.

    Let $(C_*,\partial_*)$ be the cellular chain complex of $\widetilde{X}$ (where
    $\widetilde{X}$ is equipped the cellular structure lifted from $X$). Recall that $C_*$ is a chain
    complex of free $\bbZ[\Gamma]$-modules. Since $X$ has a finite $d$-skeleton $C_j$ is a
    finitely generated free $\bbZ[\Gamma]$-module of, say, rank $e_j$ for all $ j \leq d$.  After
    choosing bases $x^{(j)}_1, \dots, x^{(j)}_{e_j} \in C_j$ 
    for the chain modules, the $\bbZ[\Gamma]$-equivariant boundary map
    $\partial_{j+1}\colon \bbZ[\Gamma]^{e_{j+1}} \to \bbZ[\Gamma]^{e_{j}}$ is given by
    right multiplication with a matrix $A_{j+1} \in M_{e_{j+1},e_{j}}(\bbZ[\Gamma])$ on row vectors for all $j < d$,
    i.e.
    $$ \partial_{j+1}(x^{(j+1)}_i) = \sum_{m=1}^{e_{j}}  a_{i,m} x^{(j)}_m.$$
    
    \begin{lemma*}
      Let $R$ be a commutative ring and let $(\phi,G)$ be a virtual pro-$p$ completion of $\Gamma$.
      Under the assumptions above, the cochain complex
      $$ \bar{C}_*: \quad 0 \longrightarrow  C^{\infty}(G,R)^{e_0} \stackrel{\bar{\partial}^0}{\longrightarrow}
      C^{\infty}(G,R)^{e_1} \stackrel{\bar{\partial}^1}{\longrightarrow} \cdots
      \stackrel{\bar{\partial}^{d-1}}{\longrightarrow} C^{\infty}(G,R)^{e_d} \longrightarrow
      0$$ computes the homology $\bar{H}^j(X;\phi,R)$ for all $j \leq d-1$ where the
      coboundary map $\bar{\partial}^j$ is multiplication with the matrix $A_{j+1}$ via the
      \emph{right} regular representation (the elements of $C^\infty(G,R)^{e_j}$ are
      considered to be column vectors).
  \end{lemma*}
  \begin{proof}
    Let $\Delta \normal_{f.i.} \Gamma$ be a finite index normal subgroup of $\Gamma$.
    Recall that the (cellular) cohomology of $\widetilde{X}/\Delta$ is
    computed by the cochain complex $\Hom_{\bbZ[\Delta]}(C_*,R)$.
    However, for every $\bbZ[\Gamma]$-module $M$ there is a canonical isomorphism
    $$\Psi\colon \Hom_{\bbZ[\Gamma]}(M,C^{\infty}(\Gamma/\Delta,R)) \to  \Hom_{\bbZ[\Delta]}(M,R)$$
    of $R$-modules where $C^{\infty}(\Gamma/\Delta,R)$ is a $\bbZ[\Gamma]$-module via the
    \emph{right} regular action.  To be more precise, the isomomorphism $\Psi$ is defined as
    $\Psi(\beta)(x) = \beta(x)(1_{\Gamma/\Delta})$ for all $ x \in M$.  In addition,
    $\Hom_{\bbZ[\Gamma]}(M,C^{\infty}(\Gamma/\Delta,R))$ is an $R[\Gamma/\Delta]$-module,
    using the action of $\Gamma/\Delta$ on $C^\infty(\Gamma/\Delta,R)$ via the \emph{left}
    regular representation; see \S\ref{par:smooth-representations}.  The $R$-module
    $\Hom_{\bbZ[\Delta]}(M,R)$ also carries an $R[\Gamma/\Delta]$-module structure by imposing
    $(\gamma \Delta \cdot \alpha)(x) = \alpha( \gamma^{-1}x)$ for all $\gamma\in \Gamma$,
    $\alpha \in \Hom_{\bbZ[\Delta]}(M,R)$ and all $x \in M$.
    Now it is easy to see that $\Psi$ is in fact an isomorphism of $R[\Gamma/\Delta]$-modules.

    Taking direct limits is exact, therefore
    the cochain complex
    $$\varinjlim_{N \normal_o G} \Hom_{\bbZ[\Gamma]}( C_*, C^\infty(\Gamma/\Gamma_N,R))$$
    computes the cohomology $\bar{H}^*(X;\phi,R)$.
    We note that $\Gamma/\Gamma_N \cong G/N$.
    By assumption, $C_j$ is a finitely generated (free) $\bbZ[\Gamma]$-module for $j \leq d$ and hence we may
    exchange $\Hom$ and the direct limit, i.e.
    $$\varinjlim_{N \normal_o G} \Hom_{\bbZ[\Gamma]}( C_j, C^\infty(G/N,R)) \cong
    \Hom_{\bbZ[\Gamma]}( C_j, C^\infty(G,R))
     \stackrel{F}{\cong} C^{\infty}(G,R)^{e_j}.$$
    The  isomorphism $F$ is constructed using our chosen bases, this means that
    $F \colon \alpha \mapsto (\alpha(x^{(j)}_1), \dots, \alpha(x^{(j)}_{e_j}))^T$. 

    Finally, we describe the coboundary maps in coordinates.
    Let $f \in C^\infty(G,R)^{e_j}$ with coordinates $(f_1,\dots,f_{e_j})^T$ and write $f = F(\alpha)$.
    Then $\bar{\partial}^j(F(\alpha)) = F(\alpha \circ \partial_{j+1})$
    and
    $$\alpha (\partial_{j+1}(x_i^{(j+1)})) = \sum_{m=1}^{e_{j}} \alpha(a_{i,m}x^{(j)}_m) =
    \sum_{m=1}^{e_{j}} \up{{r(a_{i,m})}}f_m.$$
    In other words $\bar{\partial}^j(f) = r(A_{j+1}) f$.
  \end{proof}
  \begin{remark*}
    Here it is worth noting, that $C^\infty(G,R)$ carries two commuting representations of $G$: the
    \emph{left} and the \emph{right} regular representation; see
    \S\ref{par:smooth-representations}.
    In particular, if we use a matrix $A \in M_{n,m}(\bbZ[\Gamma])$
    to act on $C^\infty(G,R)^m$ using the right regular representation, then kernel and image are
    smooth representations of $G$ via the left regular action.
  \end{remark*}

  \subsection{Definition of $p$-adic Betti numbers}
  Let $k$ be either $\bbQ$ or a finite field of characteristic $\ell \neq p$.
  Let $(\phi,G)$ be a virtual pro-$p$ completion of $\Gamma$.
 
  It follows from Lemma \ref{lem:exact-invariants} that the $k$-representation of $G$ on
  $\bar{H}^j(X;\phi,k)$ is admissible, if $H^j(\widetilde{X}/\Gamma_N; k)$ is finite
  dimensional for all $N \normal_o G$. In this case, the $p$-adic dimension
    $$ b^{[p]}_j(X;\phi,k) = \pdim^G_k \bar{H}^j(X;\phi,k) \in \bbZ_p$$
    will be called the $j$-th \emph{$p$-adic Betti number} of $X$ with respect to $\phi$ and
    with coefficients in $k$.  If $\bar{H}^j(X;\phi,k)$ is not admissible, then we impose
    $ b^{[p]}_j(X;\phi,k) = \infty$.

   \subsection{Definition of $p$-adic torsion}
   Let $(\phi,G)$ be a virtual pro-$p$ completion of $\Gamma$.  Now we consider cohomology
   with coefficients in a commutative ring $R$.  The submodule of $R$-torsion\footnote{ An
     element $\alpha$ in an $R$-module is a \emph{torsion element} if there is a
     non-zero-divisor $r \in R$ with $r\alpha = 0$.} elements $\tors\bar{H}^j(X;\phi,R)$ is
   stable under the $G$-action and will be considered as a $G$-set.  If
   $\tors \bar{H}^j(X;\phi,R)$ is an admissible $G$-set (see \S\ref{def:admissible-G-set}),
   then the $p$-adic cardinality
  $$ t^{[p]}_j(X;\phi, R) = \pcard^G \Bigl(\tors\bar{H}^j(X;\phi, R)\Bigr)  \in \bbZ_p $$
  will be called the $j$-th \emph{$p$-adic torsion} of $X$ with respect to $\phi$ and with coefficients in $R$.
  As for the Betti numbers we put $t^{[p]}_j(X;\phi,R) = \infty$ whenever
  $\tors\bar{H}^j(X;\phi,R)$ is not admissible.

  \begin{remark}
    (a) The $p$-adic Betti numbers and the $p$-adic torsion are homotopy invariants in the
    following sense.  Let $f \colon X \to Y$ be a homotopy equivalence as in
    \S\ref{par:homotopy-invariance}. It follows from Lemma~\ref{par:homotopy-invariance} that the
    are equalities $b_j^{[p]}(X;\phi,k) = b_j^{[p]}(Y;\psi,k)$ and
    $t_j^{[p]}(X;\phi,R) = t_j^{[p]}(Y;\psi,R)$ for virtual pro-$p$ completions $\phi$ and
    $\psi$ as in \eqref{diag:commuting-completions}.

    \medskip
     
     (b) For a closed subset $A \subseteq X$ we define the relative $p$-adic Betti numbers and $p$-adic torsion
     using the relative cohomology groups $\bar{H}^j(X,A;\phi,R)$.

     \medskip
    
     (c) As a case of special interest, we stress that if $X$ is a connected CW-complex with
     finite $j$-skeleton, then $\bar{H}^j(X;\phi,k)$ and the torsion part
     $\tors \bar{H}^j(X,\phi,\bbZ[\frac{1}{p}])$ are always admissible.  In particular, the $j$-th
     $p$-adic Betti number and $j$-th $p$-adic torsion are finite.

     \medskip

    (d) The $p$-adic Betti numbers and $p$-adic torsion take values in $\bbZ_p \cup \{\infty\}$, which will be
    considered as a topologically disjoint union. In particular, a sequence converges to
    $\infty$ exactly if it eventually takes the value $\infty$.  Moreover, we extend the
    addition by imposing $x + \infty = \infty + x = \infty$ for all
    $x \in \bbZ_p \cup \{\infty\}$.

    (e) As usual we define the $p$-adic invariants of a group $\Gamma$ to be the $p$-adic invariants of
    a $K(\Gamma,1)$-space.
  \end{remark}

  \subsection{Virtual invariance}\label{par:virtual-invariance}
  Let $(\phi,G)$ be a virtual pro-$p$ completion of $\Gamma$.
  For every open subgroup $K \leq_o G$, the group $\Gamma_K = \phi^{-1}(K)$
  is a finite index subgroup of $\Gamma$ and $Y = \widetilde{X}/\Gamma_K$ is a finite sheeted covering space of $X$
  with fundamental group $\Gamma_K$.
  In addition, $(\phi_{|\Gamma_K},K)$ is a virtual pro-$p$ completion of $\Gamma_K$.
  \begin{lemma*}
    In the situation described above we have
    \begin{align*}
      b_j^{[p]}(X;\phi,k) &= b_j^{[p]}(Y;\phi_{|\Gamma_K},k) \quad \text{ and }\\
      t_j^{[p]}(X;\phi,R) &= t_j^{[p]}(Y;\phi_{|\Gamma_K},R).
      \end{align*}
    \end{lemma*}
    \begin{proof}
     We observe that
     $$ \bar{H}^j(Y;\phi_{|\Gamma_K},k) \cong \bar{H}^j(X;\phi,k)_{|K}.$$
     and provided that these representations are admissible we deduce
     from \S\ref{par:res-dim} and Definition~\ref{def:p-adic-dim} that
     $$b_j^{[p]}(Y;\phi_{|\Gamma_K},k) = \pdim_k^K \Res_K^G \bar{H}^j(X;\phi,k) = \pdim_k^G \bar{H}^j(X;\phi,k).$$
     This proves the claim on Betti numbers since
     $\bar{H}^j(X;\phi,k)$ is admissible, exactly if its restriction to $K$ is admissible.
     It follows from \S\ref{par:res-card} that essentially the same argument
     applies in the case of $p$-adic torsion.
   \end{proof}

   \begin{example}\label{ex:finite-group-completion}
    We consider a topological space  $X$ with $\Gamma = \pi_1(X)$ and a subset $A \subseteq X$.
    Let $G$ be a finite group and let $\phi\colon \Gamma \to G$ be a surjective homomorphism.
    Then $(\phi,G)$ is a virtual pro-$p$ completion for every prime~$p$.  The kernel
    $\Gamma_1 = \ker(\phi)$ is a finite index subgroup of $\Gamma$.
    Recall that $A_1$ denotes the inverse image of $A$ in $\widetilde{X}/\Gamma_1$. In this case
    $\bar{H}^j(X,A;\phi,k) = H^j(\widetilde{X}/\Gamma_1,A_1;k)$ is the $j$-th relative cohomology of the finite
    sheeted covering of $X$ corresponding to $\Gamma_1$. It is equipped with the usual action
    of $\Gamma/\Gamma_1 \cong G$. Hence, the $j$-th relative $p$-adic Betti number with respect to
    $\phi$ is
    $$ b_j^{[p]}(X,A;\phi,k) = b_j(\widetilde{X}/\Gamma_1,A_1;k).$$
    Moreover, for every commutative ring $R$ there is a canonical isomorphism $\bar{H}^j(X,A;\phi,R) \cong H^j(\widetilde{X}/\Gamma_1,A_1;R)$
    and thus
    $$ t^{[p]}_j(X,A;\phi,R) = |\tors H^j(\widetilde{X}/\Gamma_1,A_1;R)|.$$
  \end{example}

  \begin{theorem}\label{thm:approximation}
    Let $X$ be a connected, locally path-connected and semilocally simply-connected
    topological space and let $A \subseteq X$ be a subset.  Let $k$ be either $\bbQ$ or
    a finite field of characteristic $\ell \neq p$ and let $R$ be a commutative ring in which
    $p$ is invertible.
    
    Let $(G,\phi)$ be a virtual pro-$p$ completion of $\Gamma = \pi_1(X)$.
    Let $(N_n)_{n \in \bbN}$ be an exhaustive chain
    in $G$. The following identities hold in $\bbZ_p \cup \{\infty\}$
    \begin{align*}
      \text{(i)} &\quad \lim_{n \to \infty} b^{[p]}_j(X,A;f_n \circ \phi, k)
                   = b_j^{[p]}(X,A;\phi,k),\\
      \text{(ii)} &\quad \lim_{n \to \infty} t_j^{[p]}(X,A;f_n\circ \phi, R)
                   = t_j^{[p]}(X,A;\phi,R),
    \end{align*}
    for all $j\in\bbN_0$ where $f_n \colon G \to G/N_n$ is the factor homomorphism.
    In particular, if all $N_n$ are open subgroups, then
    \begin{align*}
      \text{(i')} &\quad \lim_{n \to \infty} b_j(\widetilde{X}/\Gamma_{N_n},A_{N_n};k) = b_j^{[p]}(X,A;\phi,k), \\
      \text{(ii')}&\quad  \lim_{n \to \infty} |\tors H^j(\widetilde{X}/\Gamma_{N_n},A_{N_n};R)|  = t_j^{[p]}(X,A;\phi,R).
    \end{align*}
    
  \end{theorem}
  \begin{proof}
    The group $N_n$ is a pro-$p$ group for all sufficiently large $n \in \bbN$;
    this follows from the argument given in
    the proof of Lemma \ref{lem:approximation-dim}.
    Therefore, by Lemma \ref{lem:exact-invariants} (b), there is a canonical
    isomorphism
    $\bar{H}^j(X,A;f_n\circ\phi,R) \cong \bar{H}^j(X,A;\phi,R)^{N_n}$.
    Thus, if the representations $\bar{H}^j(X,A;f_n\circ\phi,k)$ (respectively the torsion submodules
    $\tors \bar{H}^j(X,A;\phi,R)$) are admissible for all large $n \in \bbN$,
    then the theorem is a direct consequence of Lemma \ref{lem:approximation-dim}
    (respectively Lemma \ref{lem:approximation-card}).

    Suppose now that there is some $e \in \bbN$ such that $N_e$ is a pro-$p$ group and
    $\bar{H}^j(X,A;f_e \circ \phi,k)$ is not admissible. In this case $\bar{H}^j(X,A;\phi,k)$ and
    $\bar{H}^j(X,A;f_{n}\circ\phi,k)$ are not admissible for all $n \geq e$.
    Indeed, there is an open subgroup
    $N_e \subseteq K \leq_o G$ such that $\bar{H}^j(X,A;f_e \circ \phi,k)^{K/N_e}$ is infinite
    dimensional.  Without loss of generality, we may assume that $K$ is pro-$p$ and by
    Lemma~\ref{lem:exact-invariants} we obtain
    $$ \bar{H}^j(X,A;f_e\circ\phi,k)^{K/N_e} \cong \bar{H}^j(X,A;\phi,k)^K  \cong \bar{H}^j(X,A; f_{n}\circ\phi,k)^{K/N_{n}} $$
    for all $n \geq e$.
    We conclude that $\bar{H}^j(X,A;\phi,k)$ and $\bar{H}^j(X,A;f_{n}\circ\phi,k)$ are not admissible.
    The same argument applies to the torsion part for cohomology with coefficients in $R$.

    Assertions (i') and (ii') follow from the discussion in Example~\ref{ex:finite-group-completion}.
  \end{proof}

  \subsection{K\"unneth formula}\label{par:products}
  Let $X_1, X_2$ be two connected, locally path-connected and semilocally simply-connected topological spaces
  and let $\Gamma_i = \pi_1(X_i)$ denote the
  corresponding fundamental group.
  Given virtual pro-$p$ completions $\phi_i\colon \Gamma_i \to G_i$ for $i \in \{1,2\}$, then
  the direct product
  $$\phi_1 \times \phi_2\colon \Gamma_1 \times \Gamma_2 \to G_1 \times G_2$$
  is a virtual pro-$p$ completion of $\Gamma_1 \times \Gamma_2$.
  \begin{proposition*}
    Let $k$ be either $\bbQ$ or a finite field of characteristic $\ell \neq p$.
    Assume that $b_j(X_i;\phi_i,k) < \infty$ for all $0 \leq j \leq n $ and $i \in \{1,2\}$.
    Then the following K\"unneth formula holds
    $$ b_n^{[p]}(X_1\times X_2;\phi_1\times\phi_2,k) = \sum_{i+j=n} b_i^{[p]}(X_1;\phi_1,k) b_j^{[p]}(X_2;\phi_2,k).$$
  \end{proposition*}
  \begin{proof}
    We fix an exhaustive chain $(N_n)_{n\in \bbN}$ of the form $N_n = N_n^{(1)} \times N_n^{(2)}$ where
    $(N_n^{(i)})_{n\in\bbN}$ is an exhaustive chain in $G_i$.
    We apply the Approximation Theorem \ref{thm:approximation} with the chain $(N_n)_{n\in\bbN}$. Since all involved
    Betti numbers are finite, the assertion follows from the K\"unneth formula for
    ordinary Betti numbers.
  \end{proof}

  \begin{proposition}[Poincar\'e duality]\label{prop:poincare-duality}
  Let $M$ be a compact, connected, orientable $n$-manifold with boundary $\partial M$.
  Decompose the boundary $\partial M = A \sqcup B$ as a disjoint union of two closed
  $(n-1)$-manifolds (where $A$ and $B$ may be empty).
  Let $(\phi,G)$ be a virtual pro-$p$ completion of $\Gamma = \pi_1(M)$. For all $j \in \{0,\dots,n\}$ the equality
  $$ b^{[p]}_j(M,A;\phi,k) = b^{[p]}_{n-j}(M,B;\phi,k).$$
  holds. Moreover, for every principle ideal domain $R$ in which $p$ is invertible,
  we have
  $$ t_j^{[p]}(M,A;\phi,R) = t_{n-j+1}^{[p]}(M,B;\phi,R).$$
  \end{proposition}
  \begin{proof}
    This is a consequence of the Approximation Theorem and
    the Poincar\'e duality isomorphism
    $$ H^j(\widetilde{M}/\Gamma_N,A_N;R) \cong H_{n-j}(\widetilde{M}/\Gamma_N,B_N;R)$$
    for the finite sheeted coverings $\widetilde{M}/\Gamma_N$ of $M$ (where $N\normal_o G$); see \cite[Thm.~3.43]{Hatcher}.

    For the $p$-adic torsion one needs to verify in addition that there is an isomorphism
    $\tors H_{j-1}(\widetilde{M}/\Gamma_N,B_N;R) \cong \tors H^{j}(\widetilde{M}/\Gamma_N,B_N;R)$ for all $j\in\bbN$.
    Indeed, observe that the universal coefficient theorem yields an isomorphism
    $$ \tors H^{j}(\widetilde{M}/\Gamma_N,B_N;R) \cong \Ext^1_R(H_{j-1}(\widetilde{M}/\Gamma_N,B_N;R),R).$$
    The homology of a compact manifold is finitely generated (see \cite[Cor.~A.9]{Hatcher})
    and since $R$ is a principle ideal domain we deduce
    $$H_{j-1}(\widetilde{M}/\Gamma_N,B_N;R) \cong R^d \oplus R/a_1R \oplus \dots R/a_tR$$
    for certain $a_1,\dots,a_t \in R \setminus\{0\}$.
    The claim follows since $\Ext_R^1(R/a_iR,R) \cong R/a_iR$ and $\Ext_R^1(R,R) = 0$.
  \end{proof}

  \subsection{The $p$-adic Euler characteristic}\label{par:eulerchar}
  Suppose that $X$ is a finite CW-complex; in particular,
  its Euler characteristic $\chi(X)$ is well-defined.
  Since all of its $p$-adic Betti numbers are finite, we may also define its $p$-adic Euler characteristic
  $$\chi^{[p]}(X;\phi) = \sum^{\dim(X)}_{j = 0} (-1)^j b_j^{[p]}(X;\phi,k)$$
  where $k$ is either $\bbQ$ or a finite field of characteristic $\ell\neq p$.
  A priori this definition depends on the chosen field $k$; however,
  as for the ordinary Euler characteristic, we will see that the choice of $k$ is inessential.
  
  Let $(N_n)_{n\in\bbN}$ be an exhaustive chain of open subgroups $N_n \normal_o G$ (which exists,
  since $G$ is finitely generated).
  From Theorem \ref{thm:approximation} we deduce that
  $$ \chi^{[p]}(X;\phi) = \lim_{n \to \infty} \chi(\widetilde{X}/\Gamma_{N_n}) = \lim_{n\to \infty} |G:N_n| \chi(X) \in \bbZ_p.$$
  The following result is an immediate consequence of equation \eqref{eq:index-limit}.
  \begin{proposition*}
    Let $X$ be a connected finite CW-complex and let $(\phi,G)$ be a virtual pro-$p$ completion.
    Then
    $$\chi^{[p]}(X;\phi) = \left\Vert G \right\Vert  \chi(X)$$
    with the notation introduced in \S\ref{par:index-convention}.
  \end{proposition*}    

  \subsection{Wedge sums}\label{par:wedge-sums}
  Let $X_1, X_2$ be two connected CW-complexes with \emph{finitely generated fundamental groups}.
  For instance, this assumption is satisfied if $X_1$ and $X_2$  have a finite $1$-skeleton.
  Let $Y = X_1 \vee X_2$ be the wedge sum subject to chosen base
  points $x_i \in X_i$. The new base point will be denoted $y_0 \in Y$.
  
  By the theorem of Seifert-van Kampen the fundamental group $\Gamma = \pi_1(Y,y_0)$ is
  isomorphic to the free product $\Gamma_1*\Gamma_2$ of $\Gamma_1 = \pi_1(X_1,x_1)$ and
  $\Gamma_2 = \pi_1(X_2,x_2)$.  Let $(\phi,G)$ be a virtual pro-$p$ completion of
  $\Gamma = \Gamma_1* \Gamma_2$.  We denote by $K_i$ the closure of $\phi(\Gamma_i)$ in $G$.
  Note that $(\phi_{|\Gamma_i},K_i)$ is a virtual pro-$p$ completion since
  $\Gamma_i$ is finitely generated by assumption.
  
  \begin{proposition*}
    Let $X_1, X_2$ be two connected CW-complexes with finitely generated fundamental groups.
    Let $R$ be a commutative ring.
    There is an exact sequence of $R$-modules with $G$-action
    \begin{equation*}
      0 \longrightarrow R \longrightarrow C^\infty(G,R) \longrightarrow I(X_1,X_2)
        \longrightarrow \bar{H}^1(X_1\vee X_2;\phi,R) \longrightarrow 0
    \end{equation*}
    where $$ I(X_1,X_2) =  \Ind_{K_1}^G \bar{H}^1(X_1,\{x_1\};\phi_{|\Gamma_1},R) \oplus \Ind_{K_2}^G \bar{H}^1(X_2,\{x_2\};\phi_{|\Gamma_2},R).$$
    Moreover, for every $j \geq 2$ there is a isomorphism
        $$ \bar{H}^j(X_1\vee X_2; \phi, R) \cong \Ind_{K_1}^G \bar{H}^j(X_1;\phi_{|\Gamma_1},R) \oplus \Ind_{K_2}^G \bar{H}^j(X_2;\phi_{|\Gamma_2},R).$$
        In particular, if $R=k$ is a field which is either $\bbQ$ or a finite field of
        characteristic $\ell \neq p$, then the following equalities of $p$-adic Betti numbers
        hold whenever all involved Betti numbers are finite:
\begin{align*}
  b_1^{[p]}(X_1 \vee X_2;\phi,k) = 1& + \left\Vert G \right\Vert - \left\Vert G:K_1\right\Vert - \left\Vert G:K_2 \right\Vert \\ &+ \left\Vert G:K_1\right\Vert b_1^{[p]}(X_1;\phi_{|\Gamma_1},k) +
                                                                    \left\Vert G:K_2 \right\Vert b_1^{[p]}(X_2;\phi_{|\Gamma_2},k)\\
  b_j^{[p]}(X_1\vee X_2;\phi,k) =&  \left\Vert G:K_1\right\Vert b_j^{[p]}(X_1;\phi_{|\Gamma_1},k) +
  \left\Vert G:K_2 \right\Vert b_j^{[p]}(X_2;\phi_{|\Gamma_2},k)
  \end{align*}
  for all $j \geq 2$.
  Here we use the notation explained in \S\ref{par:index-convention}.
  \end{proposition*}
  \begin{proof}
    We reduce to the case where $G$ is a finite group.  Indeed, taking direct limits is
    exact and induced representations of profinite groups are defined the direct limit
    of induced representations of its finite quotients; see
    \S\ref{par:induction}. Similarly, the formula for Betti numbers will follow
    from the Approximation Theorem~\ref{thm:approximation} and equation
    \eqref{eq:index-limit}.

    Let $q \colon Z \to Y = X_1 \vee X_2$ be a finite connected $G$-covering.  The restriction
    of $q$ to the inverse image $q^{-1}(X_i) = Z_i$ is a finite $G$-covering of $X_i$
    which, however, is not necessarily connected.  Let $Z_i^0$ be a connected component and
    let $K_i \leq G$ be its setwise stabilizer.  We observe that
  $$Z_i = \bigsqcup_{g \in G/K_i} gZ_i^0$$
  and deduce that
  $H^j(Z_i,q_{|Z_i}^{-1}(x_i);R) \cong \Ind_{K_i}^G H^j(Z_i^0,q_{|Z_i^0}^{-1}(x_i);R)$.  In
  fact, an isomorphism is given by the map
  $\alpha \mapsto f_\alpha$ with $f_\alpha(g) = \iota^* g^*(\alpha)$, for all $g \in G$ where
  $\iota \colon Z_i^0 \to Z_i$ denotes the inclusion.

  There are arbitrarily small
  open neighborhoods $U_i$ of $x_i$ in $X_i$ such that $x_i$ is a strong deformation retract
  of $U_i$; see the proof of Prop.~A.4 in \cite{Hatcher}.  Now the open subset $V = U_1 \vee U_2 \subseteq Y$ is contractible and we may
  choose it so small that it is regularly covered, i.e.\
  $q^{-1}(V) \cong \bigsqcup_{g\in G} gV_0$ where $V_0$ is a connected component of the fibre
  and $q_{|V_0}$ is a homeomorphism onto $V$.  It follows from the long exact sequence of the
  triple $(Z,q^{-1}(V), q^{-1}(y_0))$ that the inclusion map induces an isomorphism
  $$ H^j(Z,q^{-1}(y_0);R) \isomorph H^j(Z,q^{-1}(V);R).$$
  We apply excision twice to obtain isomorphisms
  \begin{align*}
  H^j(Z,q^{-1}(V);R) &\cong H^j(Z\setminus q^{-1}(y_0),q^{-1}(V)\setminus q^{-1}(y_0);R) \\
  &\cong H^j(Z_1, q^{-1}_{|Z_1}(U_1);R) \oplus H^j(Z_2, q^{-1}_{|Z_2}(U_2);R)
  \end{align*}
  We obtain an isomorphism
  \begin{equation}\label{eq:isomorphism-wedge-rel}
    H^j(Z,q^{-1}(y_0);R)  \isomorph H^j(Z_1, q^{-1}_{|Z_1}(x_1);R) \oplus H^j(Z_2, q^{-1}_{|Z_2}(x_2);R)
  \end{equation}
  
  for every $j$ and any commutative ring $R$.
  Since $q^{-1}_{|Z_i}(x_i)$ is just a finite discrete set of points,
  the statements for $j \geq 2$ follow readily from the long exact sequences of the pairs $(Z_i,q^{-1}(x_i))$.

  Finally, we consider the case $j=1$. The initial part of the
  long exact sequence of the pair $(Z,q^{-1}(y_0))$ reads as follows:
  $$0 \rightarrow H^0(Z;R) \longrightarrow H^0(q^{-1}(y_0);R) \longrightarrow H^1(Z,q^{-1}(y_0);R)
  \longrightarrow H^1(Z;R) \rightarrow 0.$$ Since $Z$ is connected we have $H^0(Z;R) \cong
  R$. In addition, $q^{-1}(y_0)$ is a finite discrete set of points on which $G$ acts simply
  transitively, i.e.\ $H^0(q^{-1}(y_0);R) \cong C^\infty(G;R)$.  We deduce the claimed exact
  sequence using \eqref{eq:isomorphism-wedge-rel}.

  Let $R=k$ be a field.
  The exact sequence yields
  \begin{equation}\label{eq:tmp-betti-formula}
    b_1(Z;k) = 1 - |G| + \sum_{i=1}^2 \left|G:K_i\right| b_1(Z^0_i,q^{-1}_{|Z_i^0}(x_i);k)
  \end{equation}
  Finally, we apply the long exact sequences of the pairs $(Z^0_i,q_{|Z_i^0}^{-1}(x_i))$ to see
  that
  $$ b_1(Z^0_i,q^{-1}_{|Z_i^0}(x_i);k) =  |K_i| + b_1(Z_i^0;k) -1.$$
  Substituting into formula \eqref{eq:tmp-betti-formula} yields the desired identity.
  \end{proof}

  \section{Examples}\label{sec:examples}
  Here we discuss a number of examples. Throughout $p$ denotes a prime number and $k$ a field which is either $\bbQ$ of a finite field of characteristic $\ell \neq p$.
  
   \subsection{Surfaces}\label{par:surfaces}
     Let $\Sigma_g$ be a closed oriented surface of genus $g > 0$.
    For every \emph{infinite} virtual pro-$p$ completion  $(\phi, G)$
    of $\Gamma = \pi_1(\Sigma_g)$ we have
    \begin{align*}
      b_2^{[p]}(\Sigma_g,\phi;k) &= 1,\\
      b_1^{[p]}(\Sigma_g,\phi;k) &= 2
    \end{align*}
    where the first line follows from Poincar\'e duality (Proposition\ref{prop:poincare-duality})
    and the second from the formula for
    the $p$-adic Euler characteristic \S\ref{par:eulerchar}.
    
    \subsection{Tori}\label{par:tori}
    Let $T^d$ be the $d$-dimensional torus and let $(\phi,G)$ be any virtual pro-$p$ completion
    of $\pi_1(T^d) \cong \bbZ^d$. Since every finite sheeted covering of $T^d$ is homeomorphic to $T^d$,
    the Approximation Theorem \ref{thm:approximation} implies
    $$ b_j^{[p]}(T^d,\phi;k) = \binom{d}{j} $$
    for all $j \in \{0,\dots,d\}$.

    \subsection{Free groups}\label{par:free-groups}
       Let $F_r$ be the free group of rank $r \in \bbN$ and let $(\phi,G)$ be any \emph{infinite}
       virtual pro-$p$ completion.
       Then we find that
       $$ b_1^{[p]}(F_r,\phi;k) = 1$$
       is independent of the rank $r$.  This can be deduced either from the Nielson-Schreier
       formula and the Approximation Theorem \ref{thm:approximation} or from the formula for
       the $p$-adic Euler characteristic \S\ref{par:eulerchar}.

  \subsection{Betti numbers depend on the completion (I)}\label{par:dep-completion-1}
  The $p$-adic Betti numbers depend on the virtual pro-$p$ completion $\phi$.
  This is not surprising, if one considers completions $(\phi_1,G_1)$ and $(\phi_2,G_2)$
  with $\ker(\phi_1) \neq \ker(\phi_2)$.
  
  Consider $X = S^1 \vee S^1$
  with fundamental group $\Gamma = F_2$.
  Let $\phi_1\colon\Gamma \to \widehat{F}^p_2$ be the pro-$p$ completion of $F_2$ and
  let $\phi_2\colon \Gamma \to \{1\}$ be the trivial completion.
  In this case $b^{[p]}_1(X;\phi_1,k) = 1$ (see \S\ref{par:free-groups}), whereas
  $$b^{[p]}_1(X;\phi_2,k) = b_1(X;k) = 2$$
  by Example \ref{ex:finite-group-completion}.

  \subsection{Betti numbers depend on the completion (II) }\label{par:dep-completion-2}
  In this example, we will see that the
  $p$-adic Betti numbers of a space $X$ w.r.t.\ pro-$p$ completions $\phi$ and $\psi$ can be distinct even if  $\ker(\phi) = \ker(\psi)$.

  Let $X = S^1 \times (S^1 \vee S^3)$. The space $X$ is a finite CW-complex. The fundamental
   group of $X$ is a free abelian group $\Gamma = \langle s, t \rangle \cong \bbZ^2$ of rank
   $2$, where $s$ is the class of a simple loop in the first factor and $t$ is the class of a
   simple loop in the second.

   We consider two distinct \emph{injective} virtual pro-$p$ completions of $\Gamma$.
   The first is simply the pro-$p$ completion $\phi\colon \Gamma \to \bbZ_p^2$ with $\phi(s) = (1,0)$
   and $\phi(t) = (0,1)$.
   The K\"unneth formula and the formula for wedge sums
   yield
   $$  b^{[p]}_4(X;\phi,k) = 0,$$
   see \S\ref{par:products} and \S\ref{par:wedge-sums}.
   
   For the second completion, we fix any irrational element $\omega \in \bbZ_p$
   and define $\psi = \psi_\omega \colon \Gamma \to \bbZ_p$ with $\psi(s) = \omega$ and $\psi(t) = 1$.
   Note that the irrationality of $\omega$ implies that $\psi$ is injective.
   The universal covering $\widetilde{X}$ has one $\Gamma$-orbit of $3$-cells and one $\Gamma$-orbit of $4$-cells.
   We can choose $\bbZ[\Gamma]$-bases of $C^{\text{cell}}_3(\widetilde{X})$ and
   $C^{\text{cell}}_4(\widetilde{X})$ such that 
   the boundary map $\partial_3$ is given by multiplication with $1-s \in \bbZ[\Gamma]$.

   Let $\Gamma_n = \psi^{-1}(p^n\bbZ_p)$ and let $X_n = \widetilde{X}/\Gamma_n$ be the associated finite sheeted
   covering of $X$.
   We conclude that
   $$ H_4(X_n;k) \cong k[\bbZ/p^n\bbZ]/(1-s)k[\bbZ/p^n\bbZ]$$
   where $s$ acts like addition with $\omega + p^n \bbZ_p$ on $\bbZ/p^n\bbZ$.
   A short calculation shows that
   $ b_4(X_n;k) = p^{\min(n,\nu_p(\omega))}$ and we deduce
   $$ b_4^{[p]}(X;\psi,k) = p^{\nu_p(\omega)}.$$
   It follows, in particular, that 
   $b_4^{[p]}(X;\psi,k)$ can take the value $p^m$ for every $m \geq 0$.

  \subsection{Infinite cyclic coverings}\label{par:infinite-cyclic-coverings}
  Let $X$ be a connected finite CW-complex with fundamental group $\Gamma = \pi_1(X)$.
  Let $\phi'\colon \Gamma \to \bbZ$ be an epimorphism to the infinite cyclic group. For every integer
  $m \in \bbN$ which is coprime to $p$, there is a virtual pro-$p$ completion
  $\phi_m \colon \Gamma \to \bbZ/m\bbZ \times \bbZ_p$
  obtained from $\phi'$ and the completion $\iota_m \colon \bbZ \to \bbZ/m\bbZ \times \bbZ_p$.
  In the following the profinite group $\bbZ/m\bbZ \times \bbZ_p$ is
  denoted by $G$ and the generator $\iota_m(1)$ by $t$.
  
  We consider the associated infinite cyclic covering $Y = \widetilde{X}/\ker(\phi')$ of $X$.
  The cellular chain complex $C_*(Y;k)$ of $Y$ with coefficients in $k$ is a chain complex of
  finitely generated free modules over the Laurent polynomial ring $k[t,t^{-1}]$.
  $$ C_*(Y;k)\colon \quad \cdots  C_{j+1}(Y;k) \stackrel{\partial_{j+1}}{\longrightarrow} C_{j}(Y;k) \stackrel{\partial_{j}}{\longrightarrow} C_{j-1}(Y;k) \longrightarrow \cdots$$
  We write $e_j$ for the rank of $C_j(Y;k)$.
  Since $k[t,t^{-1}]$ is a principal ideal domain, the elementary divisor theorem
  shows that we can choose bases of the chain modules $C_{j-1}$, $C_j$ and $C_{j+1}$ such that
  the boundary maps are given by right multiplication with diagonal matrices $A_j \in M_{e_j,e_{j-1}}(k[t,t^{-1}])$
  $A_{j+1} \in M_{e_{j+1},e_{j}}(k[t,t^{-1}])$
  of the form
  \begin{equation*}
    A_j = \left(\begin{array}{c | c}
                   0 & 0  \\
                  \hline
                   0 &
                             \begin{smallmatrix}
                               g_1 & & \\
                               &  \ddots & \\
                               & &  g_u
                             \end{smallmatrix}
                \end{array}\right)
     \qquad 
    A_{j+1} = \left(\begin{array}{c | c}
                  \begin{smallmatrix}
                               f_1 & & \\
                               &  \ddots & \\
                               & &  f_v
                             \end{smallmatrix} & 0\\
                      \hline
                  0 & 0
                \end{array}\right)
\end{equation*}
 with $g_u \mid g_{u-1} \mid \cdots \mid g_1$ and $f_v \mid f_{v-1} \mid \cdots \mid f_1$
 for certain non-zero $g_1, \dots, g_u, f_1,\dots,f_v \in k[t,t^{-1}]$ with $u+v \leq e_j$.
 The polynomials are uniquely determined up to units and are called the \emph{invariant factors} of $A_j$ respectively $A_{j+1}$.
  \begin{lemma*}
     In the situation described above, assume that the characteristic of $k$ does not divide $m$.
     Then the $j$-th $p$-adic Betti number w.r.t.\ $\phi_m$ is
     $$b_j^{[p]}(X;\phi_m,k) = \sum_{i=1}^v |V(f_i) \cap \mu(mp^\infty)| + \sum_{i=1}^{u}|V(g_i) \cap \mu(mp^\infty)|$$
     where $V(f_i), V(g_i) \subseteq \overline{k}^\times$ denote the vanishing sets of the Laurent polynomials
     in the algebraic closure $\overline{k}$ of $k$ and
     $$\mu(mp^\infty)=\{\zeta \in \overline{k}^\times \mid \zeta^{mp^n} = 1 \text{ for some } n \in \bbN_0\}.$$
   \end{lemma*}
   \begin{proof}
     In view of the description of the cochain complex discussed in
     \S\ref{par:cochain-complex} we see that the cochain complex
  $$ \cdots \longrightarrow C^\infty(G,k)^{e_{j-1}} \stackrel{r(A_j)}{\longrightarrow} C^\infty(G,k)^{e_j}
  \stackrel{r(A_{j+1})}{\longrightarrow} C^\infty(G,k)^{e_{j+1}} \longrightarrow \cdots$$
  computes the cohomology $\bar{H}^j(X;\phi,k)$.
  We recall that $\pdim_k^G(C^\infty(G,k)) = 0$.
  In particular, we obtain the formula
  $$ b_j^{[p]}(X;\phi,k) = \sum_{i=1}^v \pdim_k^{G}\ker(r(f_i)) - \sum_{i=1}^{u}\pdim_k^G \im(r(g_i))$$
  where $r(f_i)$ is the linear map obtained from the right regular representation on
  $C^\infty(G,k)$.  Due to the relation
  $\pdim_k^G(\ker(g_i)) + \pdim_k^G(\im(g_i)) = \pdim_k^G(C^\infty(G,k)) = 0$ it is sufficient
  to understand the $p$-adic dimensions of kernels.

  The irreducible smooth representations of $G$ over the algebraic closure
  $\overline{k}$ are one-dimensional and are parametrized by the roots of unity in
  $\mu(mp^\infty) \subseteq \overline{k}^\times$.  More precisely, for every
  $\zeta \in \mu(mp^\infty)$ there is a unique $1$-dimensional irreducible representation $S_\zeta$ such that
  $t \in G$ acts by $\zeta$.  In addition, every irreducible representation $S_\zeta$ occurs
  with multiplicity one in the right regular representation
  $C^\infty(G,\overline{k}) \cong \overline{k} \otimes_k C^\infty(G,k)$. Indeed, $S_\zeta$ is
  the space spanned by the character $\omega_\zeta \colon G \to \overline{k}^\times$ with
  $t \mapsto \zeta$.

  Let $0 \neq f \in k[t,t^{-1}]$ and write $f = \sum_{a \in \bbZ} \lambda_a t^a$ where almost
  all coefficients $\lambda_a$ vanish. How does $f$ acts on an irreducible representation
  $S_\zeta$?  A calculation shows that
  $$\up{r(f)}\omega_\zeta(x) = \sum_{a \in \bbZ} \lambda_a\omega_\zeta(x+t^{a}) = \sum_{a \in
    \bbZ} \lambda_a \zeta^{a} \omega_\zeta(x) = f(\zeta) \omega_\zeta(x)$$
  for all $x \in G$. In particular, the kernel of
  $r(f)$ consists exactly of those $S_\zeta$ for which $\zeta$ is a root of the Laurent
  polynomial $f$.  In particular, the kernel is finite dimensional and its ($p$-adic) dimension
  is exactly $|V(f) \cap \mu(mp^\infty)|$.
\end{proof}
\begin{corollary*}
  For all sufficiently large $n \in \bbN$,
  the have
  \begin{equation}\label{eq:growth-formula-cyclic}
    b_j(Y/mp^n\bbZ;\bbQ) = (mp^n)^{ b_j^{(2)}(Y,\mathcal{N}\bbZ)} + b_j^{[p]}(X;\phi_m,\bbQ)
  \end{equation}
  where $b_j^{(2)}(Y,\mathcal{N}\bbZ)$ denotes the $L^2$-Betti number of the infinite cyclic covering $Y$ with respect to the action of the infinite cyclic group.
\end{corollary*}
\begin{proof}
  In the situation described above, the $L^2$-Betti number of $Y$ with respect to the action
  by deck transformations is
  \[ b^{(2)}(Y,\mathcal{N}\bbZ)= \dim_{\bbQ(t)} \bbQ(t) \otimes_{\bbQ[t,t^{-1}]} H_j(Y,\bbQ) = e_j - u - v;\]
  see \cite[Lemma 1.34]{LückBook}.
  The sets $V(f_i) \cap \mu(mp^\infty)$ and $V(g_i) \cap \mu(mp^\infty)$ are finite, thus there
  is some $n_0 \in \bbN$ such that all elements are $mp^{n_0}$-th roots of unity.
  We leave it as an exercise to check that
  the $j$-th Betti number of the finite covering $Y/mp^n\bbZ$
  is
  $$b_j(Y/p^n\bbZ;\bbQ) = (mp^n)^{e_j-u-v} +
  \sum_{i=1}^v |V(f_i) \cap \mu(mp^n)| + \sum_{i=1}^{u}|V(g_i) \cap \mu(mp^n)|.$$
  We conclude that formula \eqref{eq:growth-formula-cyclic} holds for all $n \geq n_0$.
\end{proof}

\subsection{Knot complements and $p$-adic Betti numbers}
Let $K \subset S^3$ be a knot and let $X = S^3 \setminus \nu(K)$ be the complement of an open
tubular neighbourhood.  The fundamental group $\pi_1(X) = \Gamma$, which is called the knot
group of $K$, has an infinite cyclic abelianization; see \cite[p.112]{CrowellFox}.
The infinite cyclic covering of $X$ will be denoted by $Y$. 
We write $\phi' \colon \Gamma \to \bbZ$ to denote the abelianization map.
As in \S\ref{par:infinite-cyclic-coverings} we obtain a virtual pro-$p$ completion
$\phi_m \colon \Gamma \to \bbZ/m\bbZ \times \bbZ_p$ for every natural number $m$ which
is coprime to $p$.
\begin{proposition*}
  Let $X = S^3 \setminus \nu(K)$ be a knot complement with $\pi_1(X) = \Gamma$.
  Let $p$ be a prime number, $m \in \bbN$ coprime to $p$  and
  $\phi_m \colon \Gamma \to \bbZ/m\bbZ \times \bbZ_p$ be the associated virtual pro-$p$ completion.
  Then
  \begin{equation}\label{eq:p-adic-betti-knots}
    b^{[p]}_1(X;\phi_m,\bbQ) = 1 + \sum_{i=1}^\infty |V(\Delta_i) \cap \mu(mp^\infty)|
  \end{equation}
  where $V(\Delta_i) \subseteq \bbC^\times$ is the vanishing set of the $i$-th knot polynomial
  $\Delta_i$ of $K$; see \cite[Ch.VIII]{CrowellFox}.  For $m=1$ the first $p$-adic Betti
  number $b^{[p]}_1(X;\phi_1,\bbQ)$ equals $1$.
\end{proposition*}
\begin{proof}
  The knot complement $X$ is aspherical and is
  a $K(\Gamma,1)$-space; see \cite[3.30]{BurdeZieschang}.
  We fix a finite presentation $\Gamma = \langle \gamma_1,\dots,\gamma_s \mid w_1,\dots,w_r\rangle$.
  Since we are only interested in the first cohomology of $X$,
  we can replace $X$ by the presentation complex of the chosen presentation.
  The cellular chain complex $C_*(Y;\bbQ)$ of the
  infinite cyclic covering $Y$ is
  $$ \bbQ[t,t^{-1}]^r \stackrel{A_2}{\longrightarrow} \bbQ[t,t^{-1}]^s  \stackrel{A_1}{\longrightarrow} \bbQ[t,t^{-1}]$$
  where $A_1= (t^{\phi'(\gamma_1)}-1, \dots, t^{\phi'(\gamma_s)}-1)^T$
  and $A_2 \in  M_{r,s}(\bbQ[t,t^{-1}])$ is the Alexander matrix; see \cite[p.100]{CrowellFox}.
  The matrix $A_1$ has exactly one invariant factor: $g_1 = t-1$.
  The invariant factors $f_1,\dots,f_v$ of the Alexander matrix $A_2$
  can be normalized to satisfy
  $$\Delta_i = f_i \cdot f_{i+1} \cdots f_{v}$$
  for all $i \leq v$. By definition $\Delta_i = 1$ for all $i > v$.
  We recall that $f_v \mid f_{v-1} \mid \dots \mid f_1$ and hence
  $V(\Delta_i) = V(f_i)$. Formula \eqref{eq:p-adic-betti-knots} follows
  from the lemma in~\S\ref{par:infinite-cyclic-coverings}.

  The knot polynomials $\Delta_i$ have integral coefficients.
  Assume that $\Delta_i$ vanishes
  at some $p^n$-th root of unity.  Then the $p^n$-th cyclotomic polynomial $\Phi_{p^n}$ divides
  $\Delta_i$ in $\bbQ[t,t^{-1}]$ and by Gau{\ss}' Lemma even in $\bbZ[t,t^{-1}]$. However, this yields a contradiction since 
  $|\Delta_i(1)| = 1$ (see \cite[p.135]{CrowellFox}) and $\Phi_1(1) = 0$ and $\Phi_{p^n}(1) = p$ for all $n\geq 1$.
  We deduce that  $|V(\Delta_i) \cap \mu(p^\infty)| = 0$ and
  thus  $b^{[p]}_1(X;\phi_1,\bbQ) = 1$.
\end{proof}

The proposition and Theorem 8.21 in \cite{BurdeZieschang} show that the first $p$-adic Betti number minus one is
the limit of the ordinary Betti numbers of the $mp^n$-fold branched coverings
of $X$ as $n$ tends to $\infty$.

\begin{example*}
  The trefoil knot $K$ has knot polynomials
  $\Delta_1 = t^2 - t +1 $ and $\Delta_i = 1 $ for all $i \geq 2$.
  We observe that $\Delta_1$ is the $6$-th cyclotomic polynomial
  and we deduce that
  $$b_1^{[p]}(X;\phi_m, \bbQ) = \begin{cases}
    1 \quad &\text{ if } 6 \nmid mp\\
    3 \quad &\text{ if } 6 \mid mp
  \end{cases}.$$
\end{example*}

  \subsection{The $p$-adic torsion of (free abelian)-by-cyclic groups}\label{par:example-torsion}
  In this example we compute the $p$-adic torsion
  $t_2^{[p]}$ of certain (free abelian)-by-cyclic groups.  For instance, the fundamental
  groups of $3$-dimensional solvmanifold fall into this class.

  Let $A \in \GL_N(\bbZ)$. We study the semidirect product group
  $$ \Gamma = \bbZ^N \rtimes \bbZ,$$
  where the action on $\bbZ^N$ is defined via the matrix $A$.
  More precisely, the multiplication is given by
  $$(v,i)(w,j) = (v + A^i w, i+j).$$
   We fix a prime number $p$ and make the following assumptions:
  \begin{enumerate}
    \item[(A1)] No eigenvalue of $A$ is a root of unity and
    \item[(A2)] $A \equiv 1 \bmod p$.
   \end{enumerate}
   Assumption (A2) is convenient, but not really neccessary.  We can achieve (A2) replacing
   $A$ by $A^e$such that $A^e \equiv 1 \bmod p$ for a suitable $e \in \bbN$.  This amounts to
   passing to a finite index subgroup of $\Gamma$; by \S\ref{par:virtual-invariance} this is
   not essential for the computation of $p$-adic torsion.  Assumption (A2) implies that
   $\Gamma$ is residually a finite $p$-group.  Indeed, we use (A2) to consider $A$ as an
   element in the first principal congruence subgroup $\GL_N^1(\bbZ_p)$ of
   $\GL_N(\bbZ_p)$. Since $\GL_N^1(\bbZ_p)$ is a pro-$p$ group, we can define
   $A^\lambda \in \GL_N^1(\bbZ_p)$ for every $\lambda \in \bbZ_p$; cf.\ Lemma 4.1.1 in
   \cite{RibesZalesskii}.  It is easy to verify that the group $G = \bbZ_p^N \rtimes \bbZ_p$,
   where $\lambda \in \bbZ_p$ acts on $\bbZ_p^N$ as $A^\lambda \in \GL^1_N(\bbZ_p)$, is the
   pro-$p$ completion of $\Gamma$. Let $\phi\colon \Gamma \to G$ be the completion map.

   \subsubsection*{The $p$-adic logarithm}
   Below we shall make use of the $p$-adic logarithm and exponential map.
   The powerseries
   $$ \log(1+ B) = \sum_{j=1}^\infty \frac{(-1)^{j-1}}{j} B^j $$
   converges for all $B \in pM_N(\bbZ_p)$ and hence
   defines an analytic function $\log\colon 1+pM_N(\bbZ_p) \to pM_N(\bbZ_p)$.
   If $p$ is odd, then the $p$-adic logarithm is a homeomorphism and its inverse is given by the
   $p$-adic exponential function, which is defined using the exponential series.
   For $p=2$ the exponential series converges only on $4 M_N(\bbZ_2)$ and
   the logarithm is only a homeomorphism when restricted to $1 + 4M_N(\bbZ_2)$.

   \begin{proposition*}
     Under the assumptions (A1) and (A2) above, we have
     $$t^{[p]}_2(\Gamma;\phi,\bbZ[1/p] ) = \epsilon \det(\log(A))_{(p')} $$
     where $\epsilon = \sign \det(A - 1)$ if $p$ is odd and
     $\epsilon = \sign \det(A^{2}-1)$ if $p=2$.
     Here $x_{(p')} = x p^{-\nu_p(x)}$ is the $p'$-part of a non-zero $p$-adic number $x$.  
   \end{proposition*}
   \begin{corollary*}
     The $p$-adic torsion takes transcendental values.
   \end{corollary*}
   \begin{proof}[Proof of the Corollary]
     We only discuss the case $p > 2$ but similar examples exist also for $p=2$.
     By Hensel's Lemma there is an element $\xi \in \bbZ_p$ with $\xi^2 = p^2 + 4$.
     The matrix
     $$ A_p = \begin{pmatrix}
       1+p^2 & p \\
       p & 1\\
     \end{pmatrix}$$ has two distinct eigenvalues
     $\lambda_{\pm} = 1+ \frac{p^2}{2} \pm \frac{\xi p}{2} \in \bbZ_p$.  In particular, $A_p$
     satisfies the assumptions (A1) and (A2) above.  Since $\lambda_{+}\lambda_{-} = 1$ we
     have $0 = \log(1) = \log(\lambda_{+}) + \log(\lambda_{-})$.  We diagonalize $A_p$ to see
     that $\det(\log(A_p)) = \log(\lambda_{+}) \log(\lambda_{-}) = - \log(\lambda_{+})^2$.
     It is a result of Mahler that at most one of the $p$-adic numbers $z$ and $\exp(z)$ is algebraic
     (if $\exp$ converges at $z$); see \cite{Mahler33}.
     Since $\lambda_{+}$ is algebraic, we deduce that $\det(\log(A_p)) = -\log(\lambda_{+})^2$ is transcendental.
     Clearly, the $p'$-part of a transcendental number is transcendental
     and thus $t^{[p]}_2(\Gamma;\phi,\bbZ[1/p] )$ is transcendental.
   \end{proof}

   \begin{proof}[Proof of the Proposition]
     A short calculation shows that the commutator of $(v,e), (w,f) \in \Gamma$
     is given by
     \begin{align}\nonumber
       [(v,e),(w,f)] &= (v,e)(w,f)(-A^{-e}v,-e)(-A^{-f}w,-f)\\ 
                     &= ((A^e-1)w -(A^f-1)v,0). \label{eq:commutators}
     \end{align}
     We deduce that the subgroup $\Gamma_n = p^n\bbZ^N \rtimes p^n\bbZ$ is normal in $\Gamma$
     since $[\Gamma_n,\Gamma] \subseteq \Gamma_n$, where we use (A2) and the consequence
     $A^{p^n} \equiv 1 \bmod p^{n+1}$. In fact, $\Gamma_n = \phi^{-1}(G_n)$ where
     $G_n = p^n \bbZ_p^N \rtimes p^n \bbZ_p$.

     It follows from the universal coefficient theorem that
     $\tors H^2(\Gamma_n,\bbZ[1/p]) \cong \tors H_1(\Gamma_n,\bbZ[1/p])$; see the proof of
     Proposition \ref{prop:poincare-duality}.
     
     From \eqref{eq:commutators} one sees 
     that $[\Gamma_n,\Gamma_n] = (A^{p^n}-1)p^n\bbZ^N \rtimes\{0\}$ and we deduce
     $H_1(\Gamma_n,\bbZ) \cong \bbZ^N/(A^{p^n}-1)\bbZ^N \oplus \bbZ$.
     By assumption (A1) the matrix $A^{p^n}-1$ has full rank and therefore the
     elementary divisor theorem yields $|\tors H_1(\Gamma_n,\bbZ)| = |\det(A^{p^n}-1)|$.
     Since we are interested in the homology with coefficients in $\bbZ[1/p]$, we need to take the $p'$-part.
     Now an application of the Approximation Theorem \ref{thm:approximation}
     shows that
     \begin{equation}\label{eq:example-sol-limit}
       t_2^{[p]}(\Gamma;\phi,\bbZ[1/p]) = \lim_{n\to \infty} |\det(A^{p^n}-1)|_{(p')}.
     \end{equation}
     
     \medskip
     
     \noindent\textbf{Claim:} $\lim_{n \to \infty} p^{-n}(A^{p^n}-1) = \log(A).$\\
     We write $A = 1+B$ with $B \in pM_N(\bbZ_p)$. Then
     \begin{align*}
       p^{-n}(A^{p^n}-1) &= p^{-n} \sum_{j=1}^{p^n} \binom{p^n}{j} B^j
                         = \sum_{j=1}^{p^n} \binom{p^n-1}{j-1} \frac{1}{j} B^j\\
                        &= \sum_{j=1}^{2n} \binom{p^n-1}{j-1} \frac{1}{j} B^j + O(p^n)\\
                        &= \sum_{j=1}^{2n} \frac{(-1)\cdot(-2)\cdot(-3)\cdots(-(j-1))}{1\cdot 2 \cdots (j-1)} \frac{1}{j}B^j + O(p^n)\\
                       &= \sum_{j=1}^{2n} \frac{(-1)^{j-1}}{j}B^j + O(p^n)  = \log(A) + O(p^n)
     \end{align*}
     where we use the notation $O(p^n)$ to denote error terms in $p^nM_N(\bbZ_p)$.
     The claim follows taking the limit $n \to \infty$.

     \medskip 
   In addition, this means, that the sequence $\det(p^{-n}(A^{p^n}-1))$ converges to
   $\det(\log(A))$ as $n$ tends to $\infty$. In order to deduce that
   $\lim_{n\to\infty}\det(A^{p^n}-1)_{(p')} = \det(\log(A))_{(p')}$  we show that
   $\det(\log(A)) \neq 0$ . Assume that $p$ is odd and suppose that
   $\det(\log(A)) = 0$.  Then there is a non-zero vector $v \in \bbZ_p^N$ such that
   $\log(A)v = 0$. We deduce that $Av= \exp(\log(A))v = v$; this means, that $A$ has the
   eigenvalue $1$, which is excluded by assumption (A1).  If $p=2$, the same argument
   applies to $A^2$ using that $\log(A^2) = 2 \log(A)$.

   In order to plug $\det(\log(A))$ into formula \eqref{eq:example-sol-limit}, we need to determine
   the sign of $\det(A^{p^n}-1)$.
   Let $\lambda_1,\dots, \lambda_N\in \bbC$ be the eigenvalues (occuring with algebraic multiplicities) of $A$.
   If $\lambda_j = \overline{\lambda}_{j+1}$ is a pair of complex conjugate eigenvalues,
   then $(\lambda^{p^n}_j-1)(\lambda_{j-1}^{p^n}-1) = |\lambda^{p^n}_j-1|^2 > 0$. Therefore the
   complex eigenvalues of $A$ which are not real do not contribute to the sign of $\det(A^{p^n}-1)$.
   Suppose now that $\lambda_j$ is a real eigenvalue.
   If $p$ is odd, then
   $\lambda_j^{p^n}-1$ has the same sign as $\lambda_j-1$ and so $\sign(\det(A-1)) = \sign(\det(A^{p^n}-1))$.
   Similarly, if $p=2$, then $\lambda_j^{2^n}-1$ has the same sign as $\lambda^2_j -1$ and we
   conclude that $\sign(\det(A^2-1)) = \sign(\det(A^{2^n}-1))$.

   Finally, we use \eqref{eq:example-sol-limit} and obtain
   $$t_2^{[p]}(\Gamma;\phi,\bbZ[1/p]) = \lim_{n\to \infty} \epsilon \det(A^{p^n}-1)_{(p')} = \epsilon \det(\log(A))_{(p')}.\qedhere$$
   \end{proof}

   \section{The $p$-adic Atiyah question}\label{sec:p-adic-atiyah}
   \emph{Fix a prime number $p$.
   Throughout this section $k$ denotes either the field $\bbQ$ of rational numbers or
   a finite field of characteristic $\ell \neq p$.}
 
   \subsection{}As discussed in the introduction it is natural to wonder which values
   $p$-adic Betti numbers can actually take. Is every $p$-adic number a $p$-adic Betti number?
   In fact, we are not aware of any example where
   the $p$-adic Betti numbers take a finite value which is not an integer. Here we discuss the
   following $p$-adic analogue of the Atiyah conjecture.
   \begin{question}\label{qu:p-adic-Atiyah}
     Let $X$ be a connected finite CW-complex and $(\phi,G)$ a virtually pro-$p$ completion
     of $\Gamma = \pi_1(X)$. Under which assumptions is $b_j^{[p]}(X;\phi,k) \in \bbZ$?
   \end{question}

   \subsection{A dichotomy}
   The integrality of $p$-adic Betti numbers has a strong influence
   on the possible growth rates of Betti numbers in towers of finite sheeted covering spaces.
   Let $(N_n)_{n\in \bbN}$ be an exhaustive chain in $G$ and write $\Gamma_n = \phi^{-1}(N_n)$.
   Informally, there is the following dichotomy:\\
   \emph{Suppose that $b_j^{[p]}(X;\phi,k) \in \bbZ$ then the sequence
     $(b_j(\widetilde{X}/\Gamma_n;k))_{n\in \bbN}$ either
     \begin{enumerate}
     \item stabilizes at $b_j^{[p]}(X;\phi,k)$ or
     \item grows faster than a universal lower bound.
     \end{enumerate}}
   \noindent The universal bound will depend on the prime and the field~$k$.
   We will make this precise below under the assumption that the chain $(N_n)_{n\in \bbN}$
   is obtained from the Frattini series in a pro-$p$ group $G$.

   \subsection{The Frattini series}\label{par:Frattini-series}
   Let $G$ be a profinite group.
   The Frattini subgroup $\Phi(G)$ is the intersection of all maximal open subgroups of $G$.
   If $G$ is a pro-$p$ group, then the Frattini subgroup can be described as
   $$\Phi(G) = \overline{G^p[G,G]},$$
   i.e. $\Phi(G)$ is the smallest closed subgroup such that $G/\Phi(G)$ is an elementary abelian $p$-group; see 1.13
   in \cite{DDMS}.
   The Frattini series is the normal series of $G$ defined as
   $$ \Phi^0(G) = G \quad \text{ and } \quad \Phi^{n+1}(G) = \Phi(\Phi^n(G)).$$
   For a pro-$p$ group the Frattini series defines an exhaustive chain of open subgroups.
   Moreover, for a finite $p$-group $G$ the Frattini series eventually reaches the trivial subgroup and
   we define the \emph{Frattini length} of $G$ by
   $$\calF(G) = \min \{n \in \bbN_0 \mid \Phi^n(G) = \{1\}\}.$$

   \begin{lemma}\label{lem:frattini-basics}
     Let $G$ and $H$ be finite $p$-groups. Then
     \begin{enumerate}[ (i)]
     \item\label{it:frattini-cyclic} $\calF(\bbZ/p^r\bbZ) = r$.
     \item\label{it:frattini-ses} $\calF(G) \leq \calF(N) + \calF(G/N)$ for every normal subgroup $N \normal G$.
     \item\label{it:frattini-product} $\calF(G\times H) =  \max(\calF(G), \calF(H))$.
       \item\label{it:frattini-sub} If $H \leq G$, then $\calF(G) \leq \calF(H) + \nu_p(|G:H|)$.
     \end{enumerate}
   \end{lemma}
   \begin{proof}
     \eqref{it:frattini-cyclic} follows immediately from $\Phi(\bbZ/p^r\bbZ) \cong \bbZ/p^{r-1}\bbZ$.
     For \eqref{it:frattini-ses} let $n = \calF(G/N)$ and $m = \calF(N)$.  Then
     $\Phi^n(G) \subseteq N$ and thus $\Phi^{n+m}(G) \subseteq \Phi^m(N) = \{1\}$.  Statement
     \eqref{it:frattini-product} follows immediately from the observation
     $\Phi(G\times H) = \Phi(G) \times \Phi(H)$.  The last assertion can be obtained by
     induction on $|G:H|$.  If $|G:H| = p$, then $H$ is maximal and $\Phi(G) \subseteq
     H$. Thus $\calF(H) \geq \calF(\Phi(G)) = \calF(G) - \nu_p(|G:H|)$.  Let $|G:H| \geq
     p^2$. Take a maximal subgroup $M \leq G$ which contains $H$. Then the induction
     hypothesis yields
     \[ \calF(H) \geq \calF(M) - \nu_p(|M:H|) \geq \calF(G) - \nu_p(|G:H|).\qedhere \]
   \end{proof}

   \begin{proposition}\label{prop:frattini-bound}
     Let $p$ be a prime and let $G$ be a finite $p$-group.
     If $\rho\colon G \to \GL_N(k)$ is a faithful irreducible representation,
     then
     \begin{equation}\label{eq:frattini-bound}
       \calF(G) \leq \nu_p(N) + c_{k,p}
     \end{equation}
     with constants $c_{\bbQ,p}=1$ and $c_{\bbF_q,p} = p\log_p(q)$.
   \end{proposition}
   \begin{proof}
     We first reduce to the case where $\rho$ is primitive, this means, not induced from a
     proper subgroup.  Indeed, suppose that $\rho$ is induced from an irreducible representation of dimension
     $N'$ from a subgroup $H \leq G$, then
     $N = N' |G:H|$. Assume that \eqref{eq:frattini-bound} holds for $H$,
     then Lemma \ref{lem:frattini-basics}
     \eqref{it:frattini-sub} implies
     $$\calF(G)  \leq \calF(H) + \nu_p(|G:H|) \leq \nu_p(N') + c_{k,p} + \nu_p(|G:H|) = \nu_p(N) + c_{k,p}.$$
     Since enlarging the group, can only increase the Frattini length,
     we may further assume that $\rho(G)$ is a maximal $p$-subgroup of $\GL_N(k)$.
     The maximal primitive $p$-subgroups of $\GL_N(k)$ are well-understood going back to work of Vol'vacev.
     Here we use the complete description
     contained in \cite{LGP1986}.

     \medskip

     \emph{Assume that $p > 2$.} Let $k(\zeta_p)$ be the field obtained from adjoining a primitive
     $p$-th root of unity to $k$. By \cite[(II.4)]{LGP1986} we have $N = [k(\zeta_p)\colon k]$
     and $G$ is a cyclic group of order $p^\alpha$ where
     \[\alpha = \max\{i \mid k(\zeta_p) \text{ contains a primitive } p^{i}\text{-th root of unity}\}.\]
     Using $N < p$ and Lemma \ref{lem:frattini-basics} \eqref{it:frattini-cyclic}, we see that
     $\calF(G) = \alpha \leq \nu_p(N) + \alpha$.
     For $k = \bbQ$ we obtain $N=p-1$ and $\alpha = 1 = c_{\bbQ,p}$.  If
     $k = \bbF_q$ is a finite field, then
     $\alpha = \nu_p(q^N-1) \leq N \log_p(q) \leq p \log_p(q) = c_{\bbF_q,p}$.

     \medskip

     \emph{Assume that $p=2$ and $k = \bbQ$.}
     By \cite[(IV.4)]{LGP1986} we have $N = 1$ and $G$ is the group with $2$-elements.
     In particular, $\calF(G) = 1 \leq \nu_2(1) + 1$.

     \medskip

     \emph{Assume that $p=2$ and $k = \bbF_q$.}
     Suppose first that $q \equiv 1 \bmod 4$.
     By \cite[(III.4)(i)]{LGP1986} this implies that $N=1$ and $G$ is a cyclic group
     of order $\nu_2(q-1)$.
     Again Lemma \ref{lem:frattini-basics} \eqref{it:frattini-cyclic} yields
     $\calF(G) = \nu_2(1) + \nu_2(q-1) \leq \nu_2(1) +  2 \log_2(q)$.

     Suppose now that $q \equiv 3 \bmod 4$.
     By \cite[(III.4)(ii)]{LGP1986} there are two possibilities.
     Either $N=1$ and $|G| = 2$ (and the assertion is obvious) or
     $N = 2$ and $G$ is a semidihedral group of order $2^{\gamma+1}$ where
     $\gamma = \nu_2(q^2-1)$. In the latter case
     we obtain
     \[ \calF(G) \leq \gamma+1 = \nu_2(2) + \gamma \leq \nu_2(2) + \log_2(q^2) = \nu_2(2) + c_{\bbF_q,2}.\qedhere\]
   \end{proof}

   \begin{theorem}\label{thm:dichotomy-Frattini}
     Let $X$ be a connected, locally path-connected and semilocally simply-connected space
     with $\pi_1(X) = \Gamma$.  Let $\phi\colon \Gamma \to G$ be a completion where $G$ is a finitely generated
     pro-$p$ group. We consider $\Gamma_n = \phi^{-1}(\Phi^n(G))$; see \S\ref{par:Frattini-series}.

     Assume that $b_j^{[p]}(X;\phi,k) \in \bbZ$. Then \emph{either}
     $b_j(\widetilde{X}/\Gamma_n;k) = b_j^{[p]}(X;\phi,k)$ for all sufficiently large $n$ \emph{or}
     $$b_j(\widetilde{X}/\Gamma_n;k) \geq p^{n+1-c_{k,p}}+ b^{[p]}_j(X;\phi,k) $$
     for all sufficiently large $n$ with the constants $c_{k,p}$ as in Proposition \ref{prop:frattini-bound}. 
   \end{theorem}
   \begin{proof}
     Assume that the sequence of $j$-th Betti numbers $b_j(n) = b_j(\widetilde{X}/\Gamma_n;k)$ does not
 eventually stabilize at $b_j^{[p]} = b_j^{[p]}(X;\phi,k)$.
 Since the sequence is increasing, its values will eventually exceed the $p$-adic Betti number; i.e.,
 there is some $n_0 \in \bbN$ such that
 $b_j(n) > b_j^{[p]}$ for all $n \geq n_0$.
 The numbers $b_j(n) - b_j^{[p]}$ are positive integers (for $n \geq n_0$)
 and the product formula \cite[p.108]{Neukirch}
 implies
 \[ b_j(n)-b_j^{[p]} = |b_j(n)-b_j^{[p]}|^{-1}_p \prod_{q \neq p} |b_j(n)-b_j^{[p]}|_q^{-1} \]
 where the product runs over all primes $q \neq p$. Note that almost all terms in the product equal $1$.
 Since $b_j(n)-b_j^{[p]}$ is an integer we obtain $|b_j(n)-b_j^{[p]}|_q^{-1} \geq 1$ for every prime $q \neq p$.

 Recall that by \ref{lem:exact-invariants} (a) there is
 an isomorphism
 $H^j(\widetilde{X}/\Gamma_n;k) \cong \bar{H}^j(X;\phi,k)^{\Phi^n(G)}$.
 Decompose $\overline{H}^j(X;\phi,k)$ in irreducible constituents
 \[ \overline{H}^j(X;\phi,k) = \bigoplus_{[S] \in \Irr_k(G)} m_S \: S \]
 with certain multiplicities $m_S \in \bbN_0$.
 The representations occuring in the subrepresentation $\bar{H}^j(X;\phi,k)^{\Phi^n(G)}$ are exactly those that
 factor over $\Phi^n(G)$.
 We deduce that the difference $b_j(n) -b_j^{[p]}$ is divisible by $p^{e(n)}$
 where
 \[ e(n) = \min\bigl\{\nu_p(N) \mid \rho\colon G \to \GL_N(k) \text{ irreducible with } \Phi^n(G) \not\subseteq \ker(\rho) \bigr\}. \]
 Let $\rho\colon G \to \GL_N(k)$ be an irreducible representation which does not factor over $\Phi^n(G)$.
 The Frattini length of $\rho(G)$ is at least $n+1$. An application of Proposition \ref{prop:frattini-bound} yields
 $$ n+1 \leq \calF(\rho(G)) \leq \nu_p(N) + c_{k,p}.$$
 We conclude $|b_j(n) - b_j^{[p]}|_p \leq p^{-e(n)} \leq p^{c_{k,p} - (n+1)}$
 and this completes the proof.
\end{proof}

\subsection{Proof of Theorem \ref{thm:p-adic-analytic-towers}}
 We recall that
   $\phi \colon \Gamma \to \GL_m(\bbZ_p)$ is a homomorphism of groups and the closure
   $G = \overline{\phi(\Gamma)}$ is a compact $p$-adic Lie group of dimension $d = \dim(G)$ say.
  We  use the notation
  $\GL_m^n(\bbZ_p) =  \ker\bigl(\GL_m(\bbZ_p) \to \GL_m(\bbZ/p^n\bbZ)\bigr)$ for
  the principal congruence subgroup of level $n$.
  We study the chain
   \[ \Gamma_n = \phi^{-1}\left(\GL_m^n(\bbZ_p)\right) \]
 of principal congruence subgroups in $\Gamma$.
   The following argument is based on the theory of
  uniform pro-$p$ groups and their Lie algebras as developed in \cite[Chapter 4]{DDMS}.
  For example, the principal congruence subgroups $\GL_m^n(\bbZ_p)$ are uniform pro-$p$ groups for all $n \geq 2$.

  Let $U \normal_o G$ be an open normal uniform pro-$p$ subgroup with Lie $\bbZ_p$-algebra
  $\fu = \log(U)$; for the existence we refer to \cite[Thm.~8.1]{DDMS}.  We can choose $U$ so small that
  $U \subseteq \GL_m^2(\bbZ_p)$. Pick a natural number $r \geq 2$ such that
  $G \cap \GL_m^r(\bbZ_p) \subseteq U^p = \exp(p\fu)$.  This means, the Lie algebra
  $\fu$ satisfies $\fu \cap p^r\gl_m(\bbZ_p) \subseteq p\fu$.
   We deduce
   \[ p^{n+r} \fu \subseteq \fu \cap p^{n+r} \gl_m(\bbZ_p) \subseteq p^{n+1}\fu\]
  for every $n \in \bbN_0$.
  By Lemma~3.4 in \cite{DDMS} the Frattini series of $U$ is $\Phi^n(U) = U^{p^n} = \exp(p^n \fu)$; c.f.~\cite[Lem.~4.14]{DDMS}. In particular, we conclude
  \[ \Phi^{n+r}(U) \subseteq U \cap \GL_m^{n+r}(\bbZ_p) = G \cap \GL_m^n(\bbZ_p) \subseteq \Phi^{n+1}(U) \]
  For $\Delta_n = \phi^{-1}(\Phi^n(U))$ we observe that
$$ \Delta_{n+r} \leq \Gamma_{n+r} \leq \Delta_{n+1}$$
for all $n \in \bbN_0$.  It should be noted that the groups $\Phi^t(U)$ and
$G \cap \GL_m^{n+r}(\bbZ_p)$ are in fact pro-$p$ groups; this allows us to deduce the
inequalities
$$ b_j(\widetilde{X}/\Delta_{r+n};k) \geq  b_j(\widetilde{X}/\Gamma_{r+n};k) \geq  b_j(\widetilde{X}/\Delta_{n+1};k) $$
from Lemma \ref{lem:exact-invariants} (a).  We see, in particular, that the sequence
$(b_j(\widetilde{X}/\Delta_{n};k))_{n \in \bbN}$ stabilizes at the $p$-adic Betti number
exactly if the sequence $(b_j(\widetilde{X}/\Gamma_{n};k))_{n\in \bbN}$ does.

Assume now, that the sequences of Betti numbers do not stabilize.
Then we deduce from Theorem \ref{thm:dichotomy-Frattini} that
$$ b_j(\widetilde{X}/\Gamma_{n+r-1};k) \geq b_j(\widetilde{X}/\Delta_{n}; k) \geq c_1 p^n + b_j^{[p]}(X;\phi,k)$$
for some constant $c_1 > 0$.
Moreover, we observe that for all $n \geq 1$
\begin{align*}
  |\Gamma:\Gamma_{n+r-1}| &= |\Gamma:\Gamma_r|\; | \Gamma_r : \Gamma_{n+r-1}|
  \leq  |\Gamma:\Gamma_r|\; |\Delta_1 : \Delta_{n+r-1}|\\
  &\leq |\Gamma:\Gamma_r|\; |U : \Phi^{n+r}(U)|
  = |\Gamma:\Gamma_r|\;  p^{d(n+r)},
\end{align*}
where the last step uses Proposition~4.4 in \cite{DDMS} and the fact that the lower $p$-series
agrees with the Frattini series in uniform pro-$p$ groups. This shows that $|\Gamma:\Gamma_n|^{1/d} \leq c_2 p^n$ for some constant $c_2 > 0$ and the proof of Theorem~\ref{thm:p-adic-analytic-towers} is complete.

\subsection{An algebraic reformulation of Theorem \ref{thm:p-adic-atiyah-abelian-1}}
Consider the situation of Theorem \ref{thm:p-adic-atiyah-abelian-1}.
 By the results of \S\ref{par:cochain-complex} there is an isomorphism
 $\bar{H}^j(X;\phi,k) \cong \ker(\bar{\partial}^j)/ \im(\bar{\partial}^{j-1})$
 where
 $$ C^\infty(G,k)^{e_{j-1}} \stackrel{\bar{\partial}^{j-1}}{\longrightarrow} C^\infty(G,k)^{e_j} \stackrel{\bar{\partial}^{j}}{\longrightarrow} C^\infty(G,k)^{e_{j+1}}  $$
 and $\bar{\partial}^i = r(A_i)$ is defined as  multiplication with a matrix $A_i\in M_{e_{i+1},e_{i}}(\bbZ[\Gamma])$
 via the right regular representation.
 We obtain the following  purely algebraic
 reformulation the $p$-adic Atiyah question.
   
 \emph{Let $\Gamma$ be a finitely generated group, let $A \in M_{n,m}(\bbZ[\Gamma])$
 and
   let $(\phi,G)$ be a virtual pro-$p$ completion.
   Consider the multiplication map
   $$ r(A) \colon C^\infty(G,k)^m \to C^\infty(G,k)^n$$
   defined by $A$ using the right regular representation of $G$ on $C^\infty(G,k)$.
   The kernel of $r(A)$ is a smooth admissible representation using the left regular representation of $G$.
   Under which assumptions is $\pdim_k^G\bigl( \ker(r(A))\bigr) \in \bbZ$?}
 Theorem \ref{thm:p-adic-atiyah-abelian-1} is an immediate
 consequence of the next result.

 \begin{theorem}\label{thm:p-adic-atiyah-abelian-2}
   Let $\Gamma$ be a group and let $\phi\colon \Gamma \to G$ be a homomorphism into
   a finitely generated, virtually pro-$p$ group $G$ which is
   virtually abelian.
   For every matrix $A \in M_{m,n}(k[\Gamma])$
   the kernel of $r(A) \colon C^\infty(G,k)^n \to C^\infty(G,k)^m$
   satisfies
   $$\pdim_k^G(\ker(r(A))) \in \bbZ.$$
 \end{theorem}
 \subsection{Reduction to the case $G = \bbZ_p^d$}\label{par:reduction}
 In this section we explain why it is sufficient to prove Theorem \ref{thm:p-adic-atiyah-abelian-2}
 for the group $G = \bbZ^d_p$.
 Since $G$ is finitely generated, virtually pro-$p$ and virtually abelian, we can find an open normal
 subgroup $G_0 \normal_o G$ which is a finitely generated abelian pro-$p$ group;
 this means $G_0 \cong F \times \bbZ_p^d$ for some $d \in \bbN_0$ and a finite abelian $p$-group $F$.
 Choosing $G_0$ smaller if necessary, we may assume that $G_0 \cong \bbZ^d_p$.

 By Proposition \ref{par:res-dim} and Definition \ref{def:p-adic-dim}
 we have
 \[\pdim_k^G \ker(r(A)) = \pdim_k^{G_0} \Res_{G_0}^G(\ker(r(A))).\]
 Put $\Gamma_0 = \phi^{-1}(G_0)$. Let $u = |G:G_0| = |\Gamma:\Gamma_0|$.
 To complete the reduction step we show that there is a matrix $B \in M_{um,un}(k[\Gamma_0])$
 such that $ \Res_{G_0}^G(\ker(r(A))) \cong \ker(r(B))$.
 
 For simplicity we consider the case $n=m=1$, i.e.\ $A \in k[\Gamma]$.
 Fix a set of representatives $\gamma_1,\dots, \gamma_u \in \Gamma$ for the cosets of $\Gamma_0$.
 For $i,j \in \{1,\dots,u\}$ we write
 $\gamma_i \gamma_j = y_{i,j} \gamma_{\sigma(i)(j)}$
 with $y_{i,j} \in \Gamma_0$ and $\sigma(i)$ is a permutation of $\{1,\dots,u\}$.
 The coset representatives provide an isomorphism of (left regular) representations
 \[\Theta \colon\Res_{G_0}^G C^\infty(G,k) \to C^\infty(G_0,k)^{u} \]
 defined as $\Theta\colon f \mapsto ((\up{r(\gamma_1)}f)_{|G_0}, \dots, (\up{r(\gamma_u)}f)_{|G_0})^T$.
 
 We write $A = \sum_{j=1}^{u}x_j\gamma_j$ for certain $x_j \in k[\Gamma_0]$.
 Then the $i$-th entry of $\Theta(\up{r(A)}f)$ satisfies
 \begin{align*}
   \Theta_i(\up{r(A)}f) &= \Theta_i(\sum_{j=1}^{u} \up{r(x_j\gamma_j)}f)
              =  \Bigl(\sum_{j=1}^{u} \up{r(\gamma_i)r(x_j\gamma_j)}f\Bigr)_{|G_0}\\
             &=  \sum_{j=1}^{u} \bigl(\up{r(\gamma_ix_j\gamma_i^{-1} y_{i,j} \gamma_{\sigma(i)(j)})}f\bigr)_{|G_0}\\
             &= \sum_{s=1}^{u} \up{r(\gamma_i x_{\sigma(i)^{-1}(s)} \gamma_i^{-1} y_{i,\sigma(i)^{-1}(s)})}\Theta_s(f)      
 \end{align*}
 We write $b_{i,s} = \gamma_i x_{\sigma(i)^{-1}(s)} \gamma_i^{-1} y_{i,\sigma(i)^{-1}(s)} \in k[\Gamma_0]$,
 then the matrix $B = (b_{i,s}) \in M_{u,u}(k[\Gamma_0])$
 satisfies $\up{r(B)}\Theta(f) = \Theta(\up{r(A)}f)$ for all $f \in C^\infty(G,k)$. 
 
 \medskip

 Before we complete the proof of Theorem \ref{thm:p-adic-atiyah-abelian-2}, we collect the following Lemma.

\begin{lemma}\label{lem:mann-property}
  Let $\overline{k}$ be the algebraic closure of $k$.
  
  (a) Let $x_1,\dots, x_m \in \overline{k}^\times$ and let
  $E_1, \dots, E_m \subset (\overline{k}^\times)^d$ be subgroups of the $d$-dimensional
  algebraic torus.
  Then the $p$-adic limit
  \[\lim_{j\to\infty} \Bigl|\mu(p^j)^d \cap \bigcup_{i=1}^m x_i E_i\Bigr| \in \bbZ_p \]
  exists and lies in $\bbZ$.
  
  (b) Let $V \subseteq (\overline{k}^\times)^d$ be a Zariski closed subset and let
  $E \subseteq  (\overline{k}^\times)^d$ be a subgroup.
  Then the $p$-adic limit
  \[\lim_{j\to\infty} \Bigl|\mu(p^j)^d \cap V \cap E \Bigr| \in \bbZ_p\]
  exists and is a rational integer.
\end{lemma}
\begin{proof}
  (a) It is sufficient to consider those cosets $x_iE_i$ which contain some element from $\mu(p^\infty)^d$;
   in this case we can assume that $x_i \in \mu(p^\infty)^d$.
   We proceed by induction on $m$.  Let $m = 1$ and assume $x_1 \in \mu(p^{j_0})^d$, then
   $\mu(p^j)^d \cap x_1E_1 = \mu(p^j)^d \cap E_1$ for all $j \geq j_0$.
   Note that $E_1\cap \mu(p^\infty)^d$ is a subgroup of the $p$-torsion group $\mu(p^\infty)^d$.
   In particular, if  $E_1\cap \mu(p^\infty)^d$ is infinite, then
   \[ \lim_{j \to \infty} |\mu(p^j)^d \cap E_1| = 0 \in \bbZ_p.\]
   On the other hand, if $E_1 \cap \mu(p^\infty)^d$ is finite, then the limit is
   the order $|E_1\cap\mu(p^\infty)^d| \in \bbN$.

   Let $m \geq 2$. For the induction step we use the principle of inclusion-exclusion to see that
   $\bigl|\mu(p^j)^d \cap \bigcup_{i=1}^{m} x_i E_i\bigr|$ equals
   \begin{equation*}
   \bigl|\mu(p^j)^d \cap \bigcup_{i=1}^{m-1} x_i E_i\bigr| +
   \bigl|\mu(p^j)^d \cap x_m E_m\bigr| -  \bigl|\mu(p^j)^d \cap \bigcup_{i=1}^{m-1} (x_i E_i \cap x_m E_m)\bigr|.
   \end{equation*}
   This completes the proof, since $x_iE_i \cap x_m E_m$ is either empty or of the form $x'_i(E_i \cap E_m)$.

   (b) By a result of Poizat-Gramain the intersection
   $V\cap \mu(p^\infty)^d = \bigcup_{i=1}^m x_iE_i$ is a finite union of cosets of subgroups $E_1,\dots, E_m \subseteq (\overline{k}^\times)^d$; see  Thm.~1 and Thm.~5 in \cite{PoizatGramain06}.
   For fields of characteristic zero, this has been conjectured by Lang and has been proven by Laurent
   \cite{Laurent84}; for $\overline{k} = \bbC$ a different proof can be found in \cite{SarnakAdams94}.
   Assertion (b) follows from (a), since
   $V\cap E \cap  \mu(p^\infty)^d = \bigcup_{i=1}^m x_iE_i\cap E$
   and $x_iE_i \cap E$ is either empty or of the form $x'_i(E_i\cap E)$.
\end{proof}

\subsection{Proof of Theorem \ref{thm:p-adic-atiyah-abelian-2}}
Since the matrix $A$ has finitely many entries all of which have finite support in $\Gamma$,
we may assume that $\Gamma$ is finitely generated.  By \S\ref{par:reduction} we may assume
that $G = \bbZ_p^d$.  In addition, we can reduce to the case where $\Gamma$ is a free abelian
group.  Indeed, the group $\overline{\Gamma} = \Gamma/\ker(\phi)$ is a free abelian group and
the action of $r(A)$ only depends on the image of $A$ under the reduction homomorphism
$M_{m,n}(k[\Gamma]) \to M_{m,n}(k[\overline{\Gamma}])$.

Assume $\Gamma$ is free abelian. Observe that $\Gamma$ has rank $d+s$ for some $s \in \bbN_0$ since the image $\phi(\Gamma)$ is dense in $G = \bbZ_p^d$.
We pick a basis $t_1 ,\dots, t_{d+s}$ of $\Gamma$.
The group ring $k[\Gamma]$ is isomorphic to the Laurent polynomial ring $k[t_1^{\pm 1},\dots, t_{d+s}^{\pm 1}]$.
Since the set $\{\phi(t_1),\dots, \phi(t_{d+s})\}$ generates $\bbZ_p^d$ as $\bbZ_p$-module,
there is $d$-element subset which is a $\bbZ_p$-basis of $\bbZ_p^d$.
Without loss of generality $\phi(t_1), \dots, \phi(t_d)$ is the standard basis of $\bbZ_p^d$.

We observe that for every subgroup $G_N = p^NG$ there is an isomorphism
$\overline{k} \otimes_k C^\infty(G/G_N,k) \cong C^\infty(G/G_N,\overline{k})$ and, moreover,
the dimension of the kernel of $r(A_N)$ acting on 
acting $C^\infty(G/G_N,k)$ does not change if we extend the scalars to $\overline{k}$.

How can we describe the kernel of $r(A)$ acting on $C^\infty(G,\overline{k})^n$?
The right regular representation  $C^\infty(G,\overline{k})$
decomposes as a direct sum over irreducible representations
$$ C^\infty(G,\overline{k}) = \bigoplus_{S \in \Irr_{\overline{k}}(G)} S$$
each occuring with multiplicity one since $G$ is abelian.
The irreducible representations of $G = \bbZ^d_p$ can be parametized by the set $\mu(p^\infty)^d$ of
$d$-tuples of $p$-power roots of unity in $\overline{k}$.
For $\underline{\zeta} = (\zeta_1,\dots, \zeta_d)$
the irreducible representation $(\rho_{\underline{\zeta}},S_{\underline{\zeta}})$ is the unique $1$-dimensional $\overline{k}$-representation on which
the $i$-th basis element $\phi(t_i)$ of $\bbZ_p^d$ acts like $\zeta_i$.

We need to take a closer look in order to describe the action of $\Gamma$ on $S_{\underline{\zeta}}$.
Since the action of $\Gamma$ is defined via the action of $G$ using the homomorphism $\phi \colon \Gamma \to G$,
we get $\rho_{\underline{\zeta}}(t_i) = \zeta_i \in \overline{k}^\times$ for all $i \in \{1,\dots,d\}$.
We can write $\phi(t_{d+i}) = \sum_{j=1}^d \lambda_{i,j} \phi(t_j)$ for certain $\lambda_{i,j} \in \bbZ_p$.
Therefore $t_{d+i}$ acts as $\prod_{j=1}^{d} \zeta_{j}^{\lambda_{i,j}}$ on $S_{\underline{\zeta}}$.
The map
$$\epsilon \colon \mu(p^\infty) \to E = \Bigl\{ \underline{\zeta} \in \mu(p^\infty)^{d+s} \mid
\zeta_{d+i} = \prod_{j=1}^{d} \zeta_{j}^{\lambda_{i,j}} \text{ for all } 1\leq i\leq s \Bigr\}$$
sending $\underline{\zeta}$ to $(\zeta_1,\dots,\zeta_d, \prod_{j=1}^{d} \zeta_{j}^{\lambda_{1,j}}, \dots  \prod_{j=1}^{d} \zeta_{j}^{\lambda_{s,j}})$ is an isomorphism of groups and
 a Laurent polynomial $P \in  k[t_1^{\pm 1},\dots, t_{d+s}^{\pm 1}]$
 acts on $S_{\underline{\zeta}}$ as multiplication by $\rho_{\underline{\zeta}}(P) = P(\epsilon(\underline{\zeta}))$;
 same argument as in \S\ref{par:infinite-cyclic-coverings}.

 Now we consider a matrix $A \in M_{m,n}( k[t_1^{\pm 1},\dots, t_{d+s}^{\pm 1}])$.
 What is the dimension of the kernel of $\rho_{\underline{\zeta}}(A)$ acting on $S^n_{\underline{\zeta}}$?
 Let $J_i \subseteq k[t_1^{\pm 1},\dots, t_{d+s}^{\pm 1}]$
 be the ideal generated by all $(i \times i)$-minors of $A$.
 The rank of $\rho_{\underline{\zeta}}(A)$ is the maximal $i$ such that
 $\epsilon(\underline{\zeta}) \not\in V(J_i)$.
 We deduce that
 $$\dim_{\overline{k}} \ker\bigl(\rho_{\underline{\zeta}}(A)\bigr) = \sum_{i = 1}^n | V(J_i) \cap
 \{\epsilon(\underline{\zeta})\} |.$$

 Observe that the irreducible constituents of $C^\infty(G/G_N,\overline{k})$
 are exactly those representations $S_{\underline{\zeta}}$ with $\underline{\zeta} \in \mu(p^N)^d$, i.e.
 $$C^\infty(G/G_N,\overline{k}) \cong \bigoplus_{\underline{\zeta} \in \mu(p^N)^d} S_{\underline{\zeta}}.$$
Hence we  get the formula
$$ \dim_k \ker(r(A_N)) = \sum_{i=1}^s |V(J_i) \cap \epsilon(\mu(p^N)^d)|$$
for the dimension of the kernel of $r(A)$ acting on $C^\infty(G,k)^{G_N} = C^\infty(G/G_N,k)$.
We note that $E \cap \mu(p^N)^{d+s} = \epsilon(\mu(p^N)^d)$.
Now the proof is complete using the Approximation Lemma \ref{lem:approximation-dim}
$$ \lim_{N \to \infty} \dim_k \ker(r(A_N)) = \pdim_k^G( \ker(r(A)) )$$
and Lemma \ref{lem:mann-property} (b).

\begin{definition}
Let $X$ be a finite connected CW-complex.  We say that \emph{all $p$-adic Betti numbers of $X$ are integers} if $b_j^{[p]}(X;\phi,k) \in \bbZ$ for all $j \in \bbN_0$ and every virtual pro-$p$
completion $\phi$ of $\Gamma = \pi_1(X)$.
\end{definition}

\begin{lemma}\label{lem:wedge-sums-integers}
  Let $X_1$ and $X_2$ be finite connected CW-complexes.  All $p$-adic Betti numbers
  of $X_1$ and $X_2$ are integers if and only if all $p$-adic Betti numbers of $X_1 \vee X_2$ are integers.
\end{lemma}
\begin{proof}
  Let $\Gamma_i = \pi_1(X_i)$.
  Assume that all $p$-adic Betti numbers of $X_1$ and $X_2$ are integers.  The $p$-adic Betti
  numbers of a finite CW-complex are finite, hence we may apply the formula for
  $p$-adic Betti numbers of wedge sums in \S\ref{par:wedge-sums} to deduce that all $p$-adic Betti
  numbers of $X_1\vee X_2$ are integers.

  Conversely, assume that all $p$-adic Betti numbers of $X_1 \vee X_2$
  are finite. Let $\phi_1 \colon \Gamma_1 \to G_1$ be a virtual pro-$p$ completion. We have to
  to show $b_j^{[p]}(X_1;\phi,k)$ is an integer.
  If $G_1$ is finite there is nothing to do (see Example \ref{ex:finite-group-completion}).
  Assume that $G_1$ is infinite.
  Then $\phi \colon \Gamma_1 * \Gamma_2 \to G_1$ with $\phi_{|\Gamma_1}= \phi_1$ and
  $\phi_{|\Gamma_2} \equiv 1$ is a virtual pro-$p$ completion of $\Gamma_1 * \Gamma_2 = \pi_1(X_1\vee X_2)$.
  The formula for wedge sums in \S\ref{par:wedge-sums} yields
  \begin{equation*}
    b_j^{[p]}(X_1\vee X_2;\phi,k) =  b_j^{[p]}(X_1;\phi_1,k)
  \end{equation*}
  for all $j \in \bbN_0$. Hence all $p$-adic Betti numbers of $X_1$ are integers. The assertion for $X_2$ follows
  in the same way.
\end{proof}

\subsection{Proof of Theorem \ref{thm:p-adic-atiyah-wedge}}
Let $X = X_1 \vee \dots \vee X_n$ be a wedge sum of finite connected CW-complexes.
We assume that the fundamental groups $\Gamma_i = \pi_1(X_i)$ are virtually abelian.
We want to show that all $p$-adic Betti numbers of $X$ are integers.
In view of Lemma \ref{lem:wedge-sums-integers} it is sufficient to show that all
$p$-adic Betti numbers of $X_i$ are integers.
However,
$\Gamma_i$ is virtually abelian,
hence in every virtual pro-$p$ completion $\phi \colon \Gamma_i \to G$
the group $G$ is virtually abelian. An application of Theorem \ref{thm:p-adic-atiyah-abelian-1} completes the
proof.

\end{document}